\documentclass[reqno,12pt]{amsart} 
\usepackage{esint,mathrsfs,amsmath,amssymb,amsfonts}
\usepackage{verbatim} 
\usepackage{color}
\usepackage[hyperfootnotes=false]{hyperref}
\hypersetup{
  colorlinks,
  citecolor=blue,
  linkcolor=red,
  urlcolor=blue}
\usepackage[left=1.2in, right=1.2in, bottom=1.5in]{geometry}

\newtheorem{theorem}{Theorem}[section]
\newtheorem{lemma}[theorem]{Lemma}
\newtheorem{proposition}[theorem]{Proposition}

\theoremstyle{definition}

\newtheorem{claim}{Claim}
\theoremstyle{remark}
\newtheorem{remark}[theorem]{Remark}
\numberwithin{equation}{section}

\newcommand{\R}{\mathbb{R}}

\newcommand{\substep}[1]{\vspace{1ex}\noindent\textbf{#1}}

\begin{document}

\title[]{DIMENSION DROP FOR HARMONIC MEASURE ON AHLFORS REGULAR BOUNDARIES}

\author{Yingying Cai}
\address{Yingying Cai: Departament de Matemàtiques, Universitat Autònoma de Barcelona, Barcelona, Spain.}
\email{yingying.cai@uab.cat}
\author{Xavier Tolsa}
\address{Xavier Tolsa: ICREA, Barcelona, Departament de Matemàtiques, Universitat Autònoma de Barcelona, and Centre de Recerca Matemàtica, Barcelona, Catalonia.}
\email{xavier.tolsa@uab.cat}

\thanks{Y.C. and X.T. were supported by the European Research Council (ERC) under the European Union's Horizon 2020 research and innovation programme (grant agreement 101018680). Also partially supported by MICIU (Spain) under the grant PID2024-160507NB-I00. 
}

\begin{abstract}
We provide quantitative estimates for the dimension drop of harmonic measure. We show that for a domain $\Omega = \mathbb{R}^{n+1} \setminus E$ where $E$ is an $s$-Ahlfors regular compact set satisfying a uniform $L^2$-based non-flatness condition $\beta_2 \ge \delta_0$, the dimension of its harmonic measure is strictly less than $s$ for $s \in (n - c\delta_0^2, n]$. For planar domains, we establish an analogous quantitative threshold $s_0 = 1 - c\delta_0^2$ under Azzam's uniform non-flatness condition $\beta_\infty + \beta_{\operatorname{hole}} \ge \delta_0$.
\end{abstract}

\maketitle

\section{Introduction}

\subsection{Motivation}
The study of the relationship between harmonic measure and the geometry of the boundary is a classical topic in analysis, dating back to the work of the F. and M. Riesz \cite{Rs16}. In recent years, significant progress has been made in this area; see, e.g., \cite{A21,AHMMMTV16,AHMMT20}.


In the present paper, we focus on the dimension of harmonic measure. Let $\Omega \subset \mathbb{R}^{n+1}$ be a connected domain and $p \in \Omega$. Let $\omega^p$ denote the harmonic measure for $\Omega$ with pole $p$.  Recall that, for any Radon measure $\nu$ in $\mathbb{R}^{n+1}$, its (Hausdorff) dimension, denoted by $\dim \nu$, is defined by
$$\dim \nu = \inf \{ \dim_{\mathcal{H}} F : F \subset \mathbb{R}^{n+1} \text{ Borel},\, \nu(F^c) = 0 \},$$
where $\dim_{\mathcal{H}} F$ stands for the Hausdorff dimension of $F$.  Since $\Omega$ is connected, Harnack's inequality ensures that harmonic measures evaluated at different poles are mutually absolutely continuous. Consequently, the dimension of $\omega^p$ is independent of the choice of the pole $p$. For brevity, we will omit the superscript and simply write $\omega$.
A fundamental question is to compare $\dim \omega$ with the Hausdorff dimension of the boundary, $s = \dim_H \partial \Omega$. Naturally, one always has $\dim \omega \le s$. The behavior of $\dim \omega$ depends heavily on the dimension of the ambient space and the topology of the domain. In the planar case ($n=1$), Makarov \cite{M85} proved that for any simply connected domain, $\dim \omega = 1$, regardless of the dimension of the boundary. Later, Jones and Wolff \cite{JW88} extended this result to arbitrary planar domains, showing that $\dim \omega \le 1$. Wolff \cite{W93} sharpened this result by proving that $\omega$ must be concentrated on a set of $\sigma$-finite length.
In higher dimensions ($n \ge 2$), the situation is strictly different. Bourgain \cite{B87} proved that there exists a constant $\varepsilon_n > 0$ such that $\dim \omega \le n+1 - \varepsilon_n$ for any domain in $\mathbb{R}^{n+1}$. A natural guess would be that $\dim \omega \le n$ ($\varepsilon_n=1$ analogous to the planar case). However,  
Wolff \cite{W91} constructed ``snowflake'' domains in $\mathbb{R}^3$ (corresponding to $n=2$) where $\dim \omega>n$.
This construction was later generalized to arbitrary dimensions $n \ge 2$ by Lewis, Verchota, and Vogel \cite{LVV05}.
 A difficult open question in the area consists in finding the optimal value of the constant $\varepsilon_n$ such that $\dim \omega \le n+1-\varepsilon_n$ for any $\Omega \subset \mathbb{R}^{n+1}$. See for example \cite{J05}.

When the (Hausdorff) codimension of $\partial \Omega$ is not 1 or $\partial \Omega$ is of fractal type, many examples show that we may have $\dim \omega < \dim_{\mathcal{H}} \partial \Omega$. This is the so-called ``dimension drop'' for harmonic measure. 
This phenomenon was first observed by Carleson \cite{C85} for complements of Cantor sets and subsequently established for various classes of boundaries exhibiting self-similar or dynamical structures, such as Cantor repellers and attractors of iterated function systems; see, e.g., \cite{JW88, V91, V93,UZ02,B96,BZ15,GM05}.

In view of the results described above, it is natural to wonder if the dimension drop for harmonic measure occurs for more general, ``not dynamically generated'', subsets of $\mathbb{R}^{n+1}$ whose boundaries have fractional dimension, like domains with Ahlfors regular boundaries of fractional dimension. See Conjecture 1.9 from \cite{V22}. Recall that, given $s > 0$ and $C_0 > 1$, a compact set $E \subset \mathbb{R}^{n+1}$ is called $(s, C_0)$-Ahlfors regular if
\begin{equation}\label{ADR}
    C_0^{-1}r^s \le \mathcal{H}^s(E \cap B(x,r)) \le C_0r^s \quad \text{for all } x \in E \text{ and } 0 < r < \text{diam}(E).
\end{equation}

The question of whether dimension drop occurs for general Ahlfors regular boundaries depends subtly on the dimension $s$. Azzam \cite{A20} proved that for any $s$-Ahlfors regular boundary with $s \in (n, n+1)$, the dimension drop $\dim \omega < s$ always occurs. For the case $s \le n$ (codimension $\ge 1$), the situation is different. Using a compactness argument, Azzam \cite{A20} showed that dimension drop still occurs provided that the boundary satisfies an additional uniform non-flatness condition and $s$ is sufficiently close to $n$ (i.e., $s \in (s_0, n]$ for some threshold $s_0$ depending on the Ahlfors regularity and non-flatness constants).
On the other hand, for sufficiently small $s$, dimension drop does not always occur. David, Jeznach, and Julia \cite{DJJ23} constructed examples of $s$-Ahlfors regular sets in $\mathbb{R}^2$ for which the harmonic measure is absolutely continuous with respect to the $s$-dimensional Hausdorff measure, yielding $\dim \omega = s$. Specifically, they provided such examples contained in a line for $s \le 0.249$, and examples not contained in a line for $s$ up to slightly more than 0.4.
In the opposite direction, Tolsa \cite{T24} proved that for $s$-Ahlfors regular subsets of lines in the plane ($n=1$), if $s \in [1/2, 1)$, dimension drop always occurs.

In this paper, we focus on the quantitative aspect of Azzam's result. While \cite{A20} establishes the existence of a dimension drop for uniformly non-flat boundaries, the dependence of this drop on the underlying geometry remains implicit. Our main objective is to provide explicit and quantitative estimates in terms of the non-flatness parameter. For a recent independent work addressing this quantitative aspect, we refer the reader to the preprint \cite{P26}, which establishes bounds with an exponential dependence on the parameter; in contrast, our approach allows us to establish  quadratic bounds.

\subsection{Statement of main results} 
Throughout this paper, let $E \subset \mathbb{R}^{n+1}$ be a compact $(s, C_0)$-Ahlfors regular set. We consider the complementary domain $\Omega := \mathbb{R}^{n+1} \setminus E$ and denote by $\omega$ the associated harmonic measure with pole at infinity. We will work with the restriction of the $s$-dimensional Hausdorff measure to $E$, denoted by $\mu := \mathcal{H}^s|_E$.

To quantify the geometric non-flatness, we recall the classical Jones' beta numbers. For $x \in E$ and $0<r<\operatorname{diam}(E)$, we define the $L^\infty$-based non-flatness, the topological hole parameter, and the $L^2$-based beta number respectively by
\begin{align*}
    \beta_\infty(x,r) &:= \inf_L \sup_{y \in E \cap B(x,r)} \frac{\operatorname{dist}(y, L)}{r}, \\
    \beta_{\operatorname{hole}}(x,r) &:= \inf_L \sup_{z \in L \cap B(x,r)} \frac{\operatorname{dist}(z, E)}{r}, \\
    \beta_{2}(x,r) &:= \inf_L \left( \frac{1}{\mu(B(x,r))} \int_{B(x,r)} \left( \frac{\operatorname{dist}(y, L)}{r} \right)^2 \, d\mu(y) \right)^{1/2},
\end{align*}
where the infimum are taken over all $n$-dimensional affine planes $L \subset \mathbb{R}^{n+1}$.

\begin{theorem}\label{thm:main_higher_dim}Let $\Omega = \mathbb{R}^{n+1}\setminus E$ be a domain where $E$ is an $(s, C_0)$-Ahlfors regular compact set. Suppose that there exists $\delta_0 > 0$ such that$$\beta_{2}(x,r)\ge \delta_0, \quad \text{ for all } x\in E\text{ and }\, 0<r<\operatorname{diam}(E). \label{non-flat}$$Then there exists a constant $c > 0$, depending only on the capacity density constant  and the dimension $n$, such that for the threshold$$s_0 = n - c\delta_0^2,$$if $s \in (s_0, n]$, then $\dim \omega < s$.\end{theorem}Next, we consider the planar case.

\begin{theorem}\label{thm:main_planar}Let $\Omega = \mathbb{R}^{2}\setminus E$ be a domain where $E$ is an $(s, C_0)$-Ahlfors regular compact set. Suppose that there exists $\delta_0 > 0$ such that
\begin{equation}
\beta_\infty(x,r)+ \beta_{\operatorname{hole}}(x,r)\ge \delta_0,  \quad \text{ for all } x\in E\text{ and }\, 0<r<\operatorname{diam}(E). \label{eq:non-flat_planar}
\end{equation}
Then there exists a constant $c > 0$, depending  on $C_0$, such that for the threshold$$s_0 = 1 - c\delta_0^2,$$if $s \in (s_0, 1]$, then $\dim \omega < s$.\end{theorem}

\begin{remark}
We emphasize that both thresholds admit a lower bound that is quadratic in $\delta_0$. While the constant $c$ in Theorem \ref{thm:main_higher_dim} is independent of the Ahlfors regularity constant $C_0$, the constant in the planar case (Theorem \ref{thm:main_planar}) depends on $C_0$. This dependence arises from the control of $\beta_\infty$ by $\beta_2$ (see Lemma \ref{beta_infty}) and the compactness method utilized in Lemma \ref{6.46}.
\end{remark}

\begin{remark}
For $n\geq2$, by similar methods to the ones used to prove Theorem \ref{thm:main_planar}, for domains $\Omega\subset \R^{n+1}$
with $s$-Ahlfors regular boundaries satisfying \eqref{eq:non-flat_planar}, one can show that for
$s_0 = n - c(C_0,\varepsilon)\,\delta_0^{n+\varepsilon},$ with any $\varepsilon>0$, if $s \in (s_0, n]$, then $\dim \omega < s$.
To this end, one can use a version of Lemma \ref{beta_infty} where the $\beta_\infty$ coefficients are replaced by $\beta_p$'s,
with $p<\frac {2s}{s-2}$ if $n\geq3$, and any $p\in(1,\infty)$  if $n=2$, and then that $\beta_\infty(\frac12 Q)^{1+\frac sp}\lesssim \beta_p(Q)$.
\end{remark}

\textbf{Novelty and comparison with Azzam's result.}
Our main results improve upon and quantify the dimension drop theorem established by Azzam in two  directions:

\begin{itemize}
\item \textbf{Explicit quantification and constant independence (Theorem \ref{thm:main_higher_dim}):} 
For general dimensions, Theorem \ref{thm:main_higher_dim} replaces Azzam's  requirement with the $L^2$-based assumption $\beta_2(x,r) \ge \delta_0$. The novelty of our result is establishing the explicit threshold $s_0 = n - c\delta_0^2$ directly via the $L^2$ beta numbers, where the dimensional drop constant $c$ is independent of $C_0$. It depends only on the dimension $n$ and the capacity density constant.

\item \textbf{Explicit quantification (Theorem \ref{thm:main_planar}):} For the planar case, Theorem \ref{thm:main_planar} quantifies Azzam's result under the  $\beta_\infty + \beta_{\operatorname{hole}}$ condition. Specifically, we obtain the threshold $s_0 = 1 - c(C_0)\delta_0^2$, yielding a dimension drop with quadratic dependence on $\delta_0$. \end{itemize}

We remark that it is an open question if the dimension drop holds for arbitrary domains in $\R^{n+1}$ with $s$-Ahlfors regular boundaries with $s<n$ close enough to $n$, without non-flatness conditions analogous to the ones in Theorems 
\ref{thm:main_higher_dim} and \ref{thm:main_planar}.

\subsection{Main ideas of proof}\label{main-idea}

The general strategy to establish the dimension drop relies on proving a uniform decay estimate for the sum $\sum (\omega(P)\mu(P))^{1/2}$ over the cubes $P$ of a David-Mattila dyadic lattice. Once this decay is obtained, a standard iteration argument yields the  dimension drop. The core of the paper is therefore devoted to finding a ``stopping cube'' where the density ratio $\omega(P)/\mu(P)$ experiences a significant jump. See details in Proposition \ref{prop:stopping}, Proposition \ref{prop:stopping1} and  Section \ref{11.28}.

For Theorem \ref{thm:main_higher_dim}, we argue by contradiction. Suppose that no such density jump occurs. This assumption implies that the harmonic measure $\omega$ and the Hausdorff measure $\mu$ are mutually comparable at all intermediate scales. Under this comparability, we can utilize a  result of Tolsa (Theorem \ref{thm:Tolsa_main}), which bounds the $L^2$-based beta numbers $\beta_{2,\omega}^n$ in terms of the Riesz transform and measure densities. Crucially, when applying this theorem to the harmonic measure, we can express its Riesz transform in terms of the gradient of the Green function ($\mathcal{R}\omega = \nabla g$).  
By carefully estimating the Green function using the capacity density condition (CDC),
we bound the accumulated $L^2$-non-flatness across a window of $N$ dyadic generations. As shown in Lemma \ref{lemma:key_estimate}, this upper bound is controlled by the harmonic measure densities at the top and bottom generations of the selected window, plus a negligible error term.
However, our geometric hypothesis $\beta_2(x,r) \ge \delta_0$ acts as an opposing force: it guarantees that the underlying set is uniformly non-flat in $L^2$. We transfer this $\mu$-non-flatness to $\omega$-non-flatness, leading the accumulated sum of these beta numbers on the left-hand side must grow linearly with $N$.
A major technical obstacle here is that the density bounds naturally involve the side length $\ell(Q)^s$, and translating these into bounds relative to the measure $\mu(Q)$ inevitably introduces the Ahlfors regularity constant $C_0$. 
To eliminate this dependency, we employ a careful averaging argument (see Subsection \ref{11.57}). Ultimately, choosing the generation gap $N$ to be sufficiently large (proportional to $1/\delta_0^2$) forces a direct contradiction against the boundary density control. This proves the existence of the density jump and yields the quantitative dimension drop $s_0 = n - c\delta_0^2$ independently of $C_0$.

For the planar case (Theorem \ref{thm:main_planar}), the geometric hypothesis $\beta_\infty(x,r) + \beta_{\operatorname{hole}}(x,r) \ge \delta_0$ naturally divides our analysis into two distinct regimes. The first regime occurs when $\beta_\infty(x,r)$ is relatively large. In the planar setting $\mathbb{R}^2$, $\beta_\infty$ can be controlled by $\beta_2$ (see Lemma \ref{beta_infty}). This specific relationship is the reason Theorem \ref{thm:main_planar} is restricted to two dimensions. 
Since bounding $\beta_2$ from below by $\beta_\infty$ requires 
the Ahlfors regularity constant $C_0$, yielding the $C_0$-dependent threshold $s_0 = 1 - c(C_0)\delta_0^2$.
Once we establish a lower bound in terms of $\beta_2$, the problem reduces to the scenario resolved by Theorem \ref{thm:main_higher_dim}.
The second regime occurs when $\beta_\infty(x,r)$ is small, which consequently forces $\beta_{\operatorname{hole}}(x,r)$ to be large. To handle this, we employ a compactness argument. A crucial advantage of this approach is that the smallness of $\beta_\infty$ forces the blow-up limit of the sets to be contained within a line. Furthermore, the presence of holes in this limiting set is vital; it ensures that the domain remains locally connected across the line. This specific limiting geometry—a flat structure with hole—permits the construction of valid Harnack chains crossing the line, thereby allowing the application of Carleson estimates and the boundary Harnack inequalities. Consequently, the resulting structural constants in our estimates for this regime depend only on the Ahlfors regularity constant $C_0$, and are completely independent of the dimension parameter $s$ and the threshold $\delta_0$ (see details in Lemma \ref{6.46}).

\section{Preliminaries}

\subsection{Notations}
In this paper, constants denoted by $C$ or $c$ depend only on the dimension $n$. We write $a \lesssim b$ if there exists a constant $C>0$ such that $a \le Cb$, and $a \approx b$ if $a \lesssim b \lesssim a$. When a constant depends on additional parameters, such as the Ahlfors regular constant $C_0$ or the capacity density constant $c_{\text{cap}}$, we indicate this dependence using subscripts, for example, $\lesssim_{C_0}$, $\lesssim_{c_{\text{cap}}}$.

\subsection{Capacities and the capacity density condition}
The fundamental solution of the negative Laplacian in $\mathbb{R}^2$ is  
\[
\mathcal{E}_2(x) = \frac{1}{2\pi} \log \frac{1}{|x|},
\]  
while in higher dimensions $\mathbb{R}^{n+1}$ ($n\ge 2$) it is given by  
\[
\mathcal{E}_{n+1}(x) = \frac{c_n}{|x|^{n-1}},
\]  
with $c_n = (n-1)\mathcal{H}^n(\mathbb{S}^n)$, $\mathbb{S}^n$ being the unit sphere in $\mathbb{R}^{n+1}$.  
For a measure $\nu$ on $\mathbb{R}^{n+1}$ we define its energy  
\[
I(\nu) = \iint \mathcal{E}_{n+1}(x-y) \, d\nu(x) \, d\nu(y),
\]  
and for a set $F \subset \mathbb{R}^{n+1}$ we set  
\[
\operatorname{Cap}(F) = \frac{1}{\inf_{\nu \in M_1(F)} I(\nu)},
\]  
where the infimum is taken over all probability measures $\nu$ supported on $F$.  
In the plane this is the Wiener capacity, denoted $\operatorname{Cap}_W(F)$; for $n\ge2$ it is the Newtonian capacity. In the plane it is often more convenient to work with the logarithmic capacity  
\[
\operatorname{Cap}_L(F) = e^{-\frac{2\pi}{\operatorname{Cap}_W(F)}}.
\]

\begin{lemma}\label{5.48}
    Let $E \subset \mathbb{R}^{n+1}$ be compact and $n - 1 < s \le n + 1$. In the case $n > 1$, we have
   \[
    \operatorname{Cap}(E) \gtrsim_{s,n} \mathcal{H}_\infty^s(E)^{\frac{n-1}{s}}.
\]
In the case $n = 1$, we have
\[
    \operatorname{Cap}_L(E) \gtrsim_s \mathcal{H}_\infty^s(E)^{\frac{1}{s}}.
\]
\end{lemma}

This result is an immediate consequence of Frostman's Lemma. See \cite[Chapter 8]{M95} for the case $n > 1$, and \cite[Lemma 4]{CTV23} for the case $n = 1$, for example.

Let $\Omega \subsetneq \mathbb{R}^{n+1}$ be open, and let $\xi \in \partial\Omega$ and $r_0 > 0$. We say that $\Omega$ satisfies the $(\xi, r_0)$-local capacity density condition (CDC) if there exists some constant $c_{\mathrm{cap}} > 0$ such that, for any $r \in (0, r_0)$,$$\text{Cap}(B(\xi, r) \setminus \Omega) \ge c_{\mathrm{cap}}r^{n-1} \quad \text{in the case } n > 1,$$and$$\text{Cap}_L(B(\xi, r) \setminus \Omega) \ge c_{\mathrm{cap}} r \quad \text{in the case } n = 1.$$

We say that $\Omega$ satisfies the CDC if it satisfies the $(\xi, \text{diam}(\partial\Omega))$-local CDC uniformly for all $\xi \in \partial\Omega$. In particular, as a consequence of Lemma \ref{5.48}, if $\partial \Omega$ is $s$-Ahlfors regular for some $s > n - 1$, then $\Omega$ satisfies the CDC with the constant $c_{\text{cap}}$ depending on $C_0$.

\begin{lemma}\label{lem:CDC-Green}Let $\Omega \subsetneq \mathbb{R}^{n+1}$ be a CDC domain. Let $B$ be a ball centered on $\partial\Omega$ with $0 < r_B < \operatorname{diam}(\partial\Omega)$. Then for all $x \in \Omega \setminus 2B$ and $y \in B$, we have the following estimates for the Green function :\begin{itemize}\item In the case $n > 1$,$$G_{\Omega}(x, y) \lesssim \frac{\omega_{\Omega}^x(4B)}{\operatorname{Cap}(B\setminus\Omega)} \lesssim_{c_{\mathrm{cap}}} \frac{\omega_{\Omega}^x(4B)}{r_B^{n-1}}.$$\item In the case $n = 1$,$$G_{\Omega}(x, y) \lesssim \frac{\omega_{\Omega}^x(4B)}{\operatorname{Cap}_L(B\setminus\Omega)} \lesssim_{c_{\mathrm{cap}}} \frac{\omega_{\Omega}^x(4B)}{r_B}.$$\end{itemize}\end{lemma}

\subsection{The David–Mattila lattice}
By a standard adaptation of the dyadic lattice construction of David and Mattila \cite{DM00}, we obtain the following grid structure associated with a given compactly supported Radon measure $\mu:=\mathcal{H}^s|_{E}$.
\begin{lemma}\label{dyadic-cube}
Let $\mu$ be a compactly supported Radon measure on $\mathbb{R}^{n+1}$. Denote $E = \operatorname{supp}\mu$. 
Then there exists a dyadic grid $\mathcal{D}$ associated to $\mu$ with the constants $\ell_0$, $A_0$, $C_1$, $C$, and $\eta$ depending only on the dimension $n$, satisfying 
 the following properties:

\begin{itemize}
    \item \textbf{(Completeness)} For every $k \in \mathbb{N}$, we have $E = \bigcup_{Q \in \mathcal{D}_k} Q$.
    
    \item \textbf{(Nesting)} For every $k_0 \le k_1$, if $Q_0 \in \mathcal{D}_{k_0}$ and $Q_1 \in \mathcal{D}_{k_1}$, then either $Q_1 \subset Q_0$ or $Q_1 \cap Q_0 = \emptyset$.
    
    \item \textbf{(Tree structure)} For each $Q_1 \in \mathcal{D}_{k_1}$ and each $k_0 < k_1$, there exists a unique cube $Q_0 \in \mathcal{D}_{k_0}$ such that $Q_1 \subset Q_0$. 
    
    \item \textbf{(Scaling)} For each $Q \in \mathcal{D}_k$, there exist a point $z_Q \in Q$ and balls $\widetilde{B}_Q = B\bigl(z_Q, \tfrac{1}{2}\ell_0^k\bigr)$ and $B_Q = B\bigl(z_Q, \tfrac{A_0}{2}\ell_0^k\bigr)$ such that $\widetilde{B}_Q \cap E \subset Q \subset B_Q$.
    
    \item \textbf{(Thin boundary)} For each $Q \in \mathcal{D}_k$ and every $t > 0$, we have
    \[
    \mu\bigl(\{x \in Q : \operatorname{dist}(x, E \setminus Q) \le t \ell_0^k\}\bigr) \le C_1 t^{\eta} \mu(C B_Q)
    \]
    and
    \[
    \mu\bigl(\{x \in E \setminus Q : \operatorname{dist}(x, Q) \le t \ell_0^k\}\bigr) \le C_1 t^{\eta} \mu(C B_Q).
    \]
   
\end{itemize}
\end{lemma}

We denote the entire collection of these dyadic cubes by $\mathcal{D}_\mu = \bigcup_{k\ge0}\mathcal{D}_{k}$. By skipping generations if necessary, we may assume that the parameter $\ell_0$ satisfies $\ell_0 < 1/3$.

For any cube $Q \in \mathcal{D}_k$, we refer to the point $z_Q$ provided by the Scaling property as the center of $Q$, and we define its side length as $\ell(Q) = \tfrac{A_0}{2}\ell_0^k$.

For any cube $Q \in \mathcal{D}_\mu$ and an integer $k \ge 0$, we define the collection of its $k$-th generation descendants as
$$
    \mathcal{D}_k(Q) := \{P \in \mathcal{D}_\mu : P \subset Q \text{ and } \ell(P) = \ell_0^{k}\ell(Q)\}.
$$
More generally, for any integers $0 \le j \le k$, we denote the family of descendant cubes from the $j$-th to the $k$-th generation relative to $Q$ by
$$
    \mathcal{D}_{j,k}(Q) := \bigcup_{i=j}^k \mathcal{D}_i(Q) = \left\{ P \in \mathcal{D}_\mu : P \subset Q \text{ and } \ell(P) \in [\ell_0^k \ell(Q), \ell_0^j \ell(Q)] \right\}.
$$

Furthermore, for a constant $\lambda > 1$, we define the dilated cube $\lambda Q$ as the union of all cubes $P$ belonging to the same generation as $Q$ that satisfy $\operatorname{dist}(z_Q, P) \le \lambda\ell(Q)$. Note that this construction immediately yields the geometric inclusions$$B\bigl(z_Q, \lambda\ell(Q)\bigr)\cap E = \lambda B_Q\cap E \subset \lambda Q \subset (\lambda+2)B_Q = B\bigl(z_Q, (\lambda+2)\ell(Q)\bigr).$$Similarly, we define the $k$-th generation descendants restricted to the dilated cube $\lambda Q$ as$$\mathcal{D}_k(\lambda Q) = \{P \in \mathcal{D}_\mu : P \subset \lambda Q \text{ and } \ell(P) = \ell_0^{k}\ell(Q)\}.$$

\subsection{Riesz transforms and $\beta_2^n$ coefficients}

The following theorem, due to Tolsa \cite{T25}, plays a crucial role in our argument.

\begin{theorem}[Tolsa \cite{T25}]\label{thm:Tolsa_main}
Let $\nu$ be a Radon measure in $\mathbb{R}^{n+1}$ and let $B_0 \subset \mathbb{R}^{n+1}$ be a closed ball. Suppose that $M_n(\chi_{3B_0}\nu) \in L^2(\nu|_{2B_0})$ and $\mathcal{R}_*(\chi_{3B_0}\nu) \in L^2(\nu|_{2B_0})$. Then
\[
\begin{aligned}
&\int_{B_0} |M_n(\nu|_{B_0})|^2 \, d\nu + \int_{B_0} \int_0^{2\operatorname{rad}(B_0)} \beta_{2,\nu}^n(x,r)^2 \Theta^n_\nu(x,r) \frac{dr}{r} \, d\nu(x) \\
&\qquad \lesssim \int_{2B_0} |\mathcal{R}\nu - m_{\nu,2B_0}(\mathcal{R}\nu)|^2 \, d\nu + \mathcal{P}_\nu(B_0)^2 \nu(2B_0) + \int_{2B_0} \theta_{\nu}^{n,*}(x)^2 \, d\nu(x),
\end{aligned}
\]
where the implicit constant depends only on $n$.
\end{theorem}

For the reader's convenience, we briefly recall the notation involved. For a Radon measure $\nu$ on $\mathbb{R}^{n+1}$, $x \in \mathbb{R}^{n+1}$, $r > 0$, and a dimension parameter $d \in (0, n+1]$ (which will typically be $n$ or $s$ in our setting), the $d$-dimensional density of $\nu$ on $B(x,r)$ is defined by$$\Theta_\nu^d(x,r) = \frac{\nu(B(x,r))}{r^d}.$$For $1 \le p < \infty$, the corresponding $L^p$-based $d$-dimensional Jones' beta number is given by$$\beta_{p,\nu}^d(x,r) = \inf_L \left( \frac{1}{r^d} \int_{B(x,r)} \left( \frac{\operatorname{dist}(y, L)}{r} \right)^p d\nu(y) \right)^{1/p},$$where the infimum is taken over all $n$-dimensional affine planes $L \subset \mathbb{R}^{n+1}$. Associated with the $n$-dimensional density, we define the maximal density operator and the upper $n$-dimensional density respectively as$$M_n\nu(x) = \sup_{r>0} \Theta_\nu^n(x,r), \qquad \theta_\nu^{n,*}(x) = \limsup_{r\to 0} \Theta_\nu^n(x,r).$$Additionally, for a ball $B = B(x,r)$, we define sum$$\mathcal{P}_\nu(B) := \sum_{j \ge 0} 2^{-j} \Theta_\nu^n(x, 2^j r).$$Finally, we recall the Riesz transform operators. The $n$-dimensional Riesz transform of a  Radon measure $\nu$ is defined formally by$$\mathcal{R}\nu(x) = \int \frac{x-y}{|x-y|^{n+1}} \, d\nu(y).$$For $\varepsilon > 0$, the $\varepsilon$-truncated Riesz transform is$$\mathcal{R}_\varepsilon\nu(x) = \int_{|x-y|>\varepsilon} \frac{x-y}{|x-y|^{n+1}} \, d\nu(y).$$The maximal Riesz transform and the principal value are defined respectively as$$\mathcal{R}_*\nu(x) = \sup_{\varepsilon>0} |\mathcal{R}_\varepsilon\nu(x)|, \qquad \operatorname{pv} \mathcal{R}\nu(x) = \lim_{\varepsilon\to 0} \mathcal{R}_\varepsilon\nu(x),$$whenever the latter limit exists. By a standard abuse of notation, we will frequently write $\mathcal{R}\nu$ to denote the principal value $\operatorname{pv} \mathcal{R}\nu$.

\section{Existence of a stopping cube}
In this section we prove Proposition \ref{prop:stopping}, which is the key to obtain the dimension drop. 
The proof proceeds by contradiction. Assuming that no stopping cube exists, we derive an estimate that leads to a contradiction after a careful averaging argument and an application of Theorem \ref{thm:Tolsa_main}.

\begin{proposition}\label{prop:stopping}
Under the assumption \eqref{non-flat},
there exists an integer $m_0$ depending on $\delta_0$ and $C_0$ such that for every cube $Q\in\mathcal{D}_\mu$ one can find a descendant $P\in\mathcal{D}_{0,m_0}(Q)$ satisfying  either
\[
\frac{\omega(P)}{\mu(P)}\;\ge\; 16\,\frac{\omega(Q)}{\mu(Q)}\qquad\text{or}\qquad 
\frac{\omega(P)}{\mu(P)}\;\le\; 16^{-1}\,\frac{\omega(Q)}{\mu(Q)}.
\]  
\end{proposition}

\begin{proof}
We argue by contradiction. Suppose that for some cube $\tilde{Q}_0\in\mathcal{D}_\mu$ and for every $P\in\mathcal{D}_{0,m_0}(\tilde{Q}_0)$ we have  
\begin{equation}\label{eq:ratio_bound}
16^{-1}\,\frac{\omega(\tilde{Q}_0)}{\mu(\tilde{Q}_0)}\;<\;\frac{\omega(P)}{\mu(P)}\;<\;16\,\frac{\omega(\tilde{Q}_0)}{\mu(\tilde{Q}_0)}.
\end{equation}

\subsection{Setup and auxiliary objects}
In what follows, we fix a cube $Q_0 \in \mathcal{D}_\mu$ such that $100 Q_0 \subset \tilde{Q}_0$ and $\ell(Q_0) \approx \ell(\tilde{Q}_0)$. We then choose an integer $j$ sufficiently large such that every descendant cube $Q \in \mathcal{D}_j(Q_0)$ satisfies $200Q \subset \tilde{Q}_0$. Fixing such a cube $Q$, let $B_Q$ be its associated ball given by the Scaling property in Lemma \ref{dyadic-cube}. In order to apply Theorem \ref{thm:Tolsa_main}, we take the ball $B_0$ in the theorem as $B_0 := 49B_Q$. Accordingly, for a parameter $t \approx \ell_0^{N}\ell(Q)(N$ specified later), we define the regularized measure
\[
\nu := \omega|_{3B_0} * \varphi_t .
\]

Similar to the definition for balls, for a dyadic cube $P \in \mathcal{D}_\mu$, a Radon measure $\nu$ (which will typically be either $\mu$ or $\omega$ in our setting), and a dimension parameter $d \in (0, n+1]$, we define the discrete $L^p$-based Jones' beta number for $1 \le p < \infty$ as$$\beta_{p,\nu}^d(P) = \inf_L \left( \frac{1}{\ell(P)^d} \int_{P} \left( \frac{\operatorname{dist}(y, L)}{\ell(P)} \right)^p d\nu(y) \right)^{1/p},$$where the infimum is again taken over all $n$-dimensional affine planes $L \subset \mathbb{R}^{n+1}$.

The following lemma provides the quantitative transfer mechanism between $\beta_{2,\mu}^s$ and $\beta_{2,\omega}^n$, which is necessary to bridge our geometric hypothesis \eqref{non-flat} and the  Theorem \ref{thm:Tolsa_main}.

\begin{lemma}\label{lower-bound}
Fix $\epsilon\in (0,1)$ and $m_0=m_0(\epsilon, N)$ large enough. Then we have \begin{equation}
\begin{aligned}
\sum_{P\in\mathcal{D}_{0,N}(Q)}\Theta_\omega^n (P)^2 \,\frac{\omega(P)}{\mu(P)} &\,\Big( \beta_{2,\mu}^s (P)^2 \ell(P)^s-CC_0^2\epsilon^2
\mu(P)\Big)\\
&\lesssim \int_{4B_Q} \int_t^{12\ell(Q)} \beta_{2,\omega}^n(x,r)^2 \,\Theta_\nu^n(x,2r)\,\frac{dr}{r}\,d\nu(x).
\end{aligned}
\end{equation}

\end{lemma}

\begin{proof}
To establish the lemma, it suffices to prove a localized estimate for each cube $P \in \mathcal{D}_{0,N}(Q)$. Moreover, by an averaging argument, there exist $x \in 2B_{P}$ and $r \in (10\ell(P), 12\ell(P))$ such that$$\int_{2B_{P}}\int_{10\ell(P)}^{12\ell(P)}\beta_{2,\omega}^n (y,\rho)^2 \Theta_\nu^n (y, 2\rho)\frac{d\rho}{\rho}d\nu(y) \approx\beta_{2,\omega}^n (x,r)^2 \Theta_\nu^n (x, 2r) \nu( 2B_{P}).$$Observe that $\Theta_\nu^n (x, 2r) \nu( 2B_{P}) \gtrsim \Theta_\omega^n (P)\omega(P) $.
Specifically, since $P \subset B(x,r/2)$ and $\mu$ is $(s,C_0)$-Ahlfors regular, the desired inequality follows if we can show that
\begin{equation}\label{1,00}
 \frac{\omega(P)}{\mu(P)}\Big(
\beta_{2,\mu}^s(x,r/2)^2 r^s-C\epsilon^2\mu(B(x,r))\Big)\lesssim \beta_{2,\omega}^n(x,r)^2 r^n= \beta_{2,\omega}^s (x,r)^2 r^s.
\end{equation}

Denote by \(L\) an \(n\)-plane that minimises \(\beta_{2,\omega}^s(x,r)\) (the existence of such a plane is standard).  
Consider a Whitney decomposition of \(E\setminus L\) with respect to \(L\) (i.e., we write \(E\setminus L\) as a disjoint union of maximal dyadic cubes \(R\in\mathcal{D}_\mu\) such that  
\(\operatorname{dist}(R,L)\;\approx\; \ell(R)\). Denote by $\mathcal{W}_{x,r}$ the family of those Whitney cubes such that
$R\cap B(x,r/2)\neq\varnothing$. By choosing appropriately the parameters of the Whitney decomposition we can assume
that all the cubes in $\mathcal{W}_{x,r}$ are contained in $B(x,r)$.
Then we have \begin{equation}\begin{aligned}
\beta_{2,\omega}^s (x,r)^2 r^s&=\int_{B(x,r)\cap E} \Big(\frac{\operatorname{dist}(y,L)}{r}\Big)^2 d\omega(y)\\&= \int_{B(x,r)\cap E\setminus L} \Big(\frac{\operatorname{dist}(y,L)}{r}\Big)^2 d\omega(y)\gtrsim \sum_{R\in \mathcal{W}_{x,r}}\Big(\ell(R)/\ell(P)\Big)^2 \omega(R).\end{aligned}\end{equation}
Similarly, recalling the definition of $\beta_{2,\mu}^s (x,r/2)$, we obtain$$\beta_{2,\mu}^s (x,r/2)^2 r^s \lesssim \sum_{R\in \mathcal{W}_{x,r}}\Big(\ell(R)/\ell(P)\Big)^2 \mu(R).$$Split $\mathcal{W}$ into two subfamilies:$$\mathcal{I}_1 = \bigl\{R\in\mathcal{W}_{x,r}: \ell(R)\ge \epsilon \ell(P)\bigr\},\qquad
\mathcal{I}_2 = \bigl\{R\in\mathcal{W}_{x,r}: \ell(R)<  \epsilon\ell(P)\bigr\}.$$For the cubes in $\mathcal{I}_2$ we have the trivial estimate$$\sum_{R\in\mathcal{I}_2} \Big(\ell(R)/\ell(P)\Big)^{\!2}\mu(R)\;\le\; \epsilon^2\,\mu(B(x,r)).$$For the cubes in $\mathcal{I}_1$, by \eqref{eq:ratio_bound}, we have$$\sum_{R\in\mathcal{I}_1} \Big(\ell(R)/\ell(P)\Big)^{\!2}\mu(R)=
\sum_{R\in\mathcal{I}_1} \Big(\ell(R)/\ell(P)\Big)^{\!2}\omega(R)\frac{\mu(R)}{\omega(R)}\approx \frac{\mu(P)}{\omega(P)}\sum_{R\in\mathcal{I}_1} \Big(\ell(R)/\ell(P)\Big)^{\!2}\omega(R).$$

Combine the above estimate, we have 
\[
\begin{aligned}
\beta_{2,\mu}^s (x,r/2)^2 r^s & \lesssim \sum_{R\in \mathcal{W}_{x,r}}\Big(\ell(R)/\ell(P)\Big)^2 \mu(R)\\
&\lesssim\sum_{R\in\mathcal{I}_2} \Big(\ell(R)/\ell(P)\Big)^{\!2}\mu(R)+ \sum_{R\in\mathcal{I}_1} \Big(\ell(R)/\ell(P)\Big)^{\!2}\mu(R)\\
 &\lesssim \epsilon^2\,\mu(B(x,r))+
 \frac{\mu(P)}{\omega(P)}\sum_{R\in\mathcal{I}_1} \Big(\ell(R)/\ell(P)\Big)^{\!2}\omega(R)\\
 &\lesssim \epsilon^2\,\mu(B(x,r))+
 \frac{\mu(P)}{\omega(P)}\beta_{2,\omega}^s (x,r)^2 r^s.
\end{aligned}
\]
Multiplying both sides by $\frac{\omega(P)}{\mu(P)}$, we arrive at \eqref{1,00}.
\end{proof}

Having established the generalized estimate in Lemma \ref{lower-bound}, we now specialize it by applying our  geometric assumption.
\begin{lemma}\label{lower-bound1}
 Under the above assumptions and assumption \eqref{non-flat}, then we have 
   \[
     \delta_0^2 \ell(Q)^{2(s-n)}\, \!\!\sum_{P\in\mathcal{D}_{0,N}(Q)}\!\!\Theta_\omega^s(P)^2\omega(P)
     \lesssim \int_{4B_Q} \int_t^{12\ell(Q)} \beta_{2,\omega}^n(x,r)^2 \,\Theta_\nu^n(x,2r)\,\frac{dr}{r}\,d\nu(x).
     \]
\end{lemma}
\begin{proof}The proof follows the same localized averaging argument as in Lemma \ref{lower-bound}. By repeating this argument for each cube $P \in \mathcal{D}_{0,N}(Q)$, we can find a point $x \in 2B_{P}$ and a radius $r \in (10\ell(P), 12\ell(P))$ such that the upper bound by the integral holds. It then suffices to estimate the corresponding localized quantity on the left-hand side. Under the  assumption \eqref{non-flat}, we then  
take $\epsilon = c(C_0)\delta_0$ for an appropriately small constant $c(C_0)$, we can absorb the negative error term. Since $P \subset B(x,r/2)$ and $\mu$ is an Ahlfors regular measure, this yields$$\beta_{2,\mu}^s (x,r/2)^2 r^s - C \epsilon^2\,\mu(B(x,r)) \gtrsim \delta_0^2\mu(B(x,r/2)) \gtrsim \delta_0^2\mu(P).$$Multiplying this lower bound by the density ratio $\frac{\omega(P)}{\mu(P)}$ (which appears in the sum as shown in Lemma \ref{lower-bound}), the $\mu(P)$ term cancels out, leaving us with a lower bound proportional to $\delta_0^2 \omega(P)$.

Finally, we substitute this localized lower bound back into the sum. Noting that $\Theta_\omega^n(P)^2 = \Theta_\omega^s(P)^2 \ell(P)^{2(s-n)}$ and $\ell(P) \le \ell(Q)$ for $P \in \mathcal{D}_{0,N}(Q)$, we can pull the factor $\ell(Q)^{2(s-n)}$ out of the summation, which establishes the desired inequality.\end{proof}

\subsection{Key estimate for $l_j(\omega)$}

In the preceding subsection, we bounded the quantities arising from Theorem \ref{thm:Tolsa_main} from below. The current subsection establishes the upper bound, yielding the key estimate presented in Lemma \ref{lemma:key_estimate}.

To achieve this, our strategy relies on the connection between the harmonic measure $\omega$ and the Green function $g$. This allows us to translate the problem of bounding the harmonic measure into controlling the Green function. Furthermore, the Thin boundary property from Lemma \ref{dyadic-cube} plays a crucial role in bounding the error terms that arise near the boundaries of the cubes.

For each generation \(j\) we introduce
\[
l_j(\omega):=\sum_{P\in\mathcal{D}_{j,j+N}(Q_0)}\Theta_\omega^s(P)^2\,\omega(P)\quad \text{ and }\quad
l_j(\mu):=\sum_{P\in\mathcal{D}_{j,j+N}(Q_0)}\Theta_\mu^s(P)^2\,\mu(P).
\]

\begin{lemma}\label{lemma:key_estimate}
Under the assumptions above, if $N\leq \frac1{n-s}$, we have
\begin{equation}\label{eq:key_estimate}
\delta_0^2 \, l_j(\omega) \;\lesssim_{c_{\mathrm{cap}}}\;
\sum_{Q \in \mathcal{D}_j(Q_0)} \Theta_\omega^s(Q)^2 \,\omega(Q)
\;+\; \sum_{Q \in \mathcal{D}_{j+N}(Q_0)} \Theta_\omega^s(Q)^2 \,\omega(Q)
\;+\; C^* \, \ell_0^{j\eta} \, \theta_\omega^s(Q_0)^2 \,\omega(Q_0),
\end{equation}
where \(C^*\) depends on \(C_0\), and \(\eta>0\) is as given in Lemma \ref{dyadic-cube}.
\end{lemma}

\begin{proof}
We fix an arbitrary cube \(Q \in \mathcal{D}_j(Q_0)\) and work on a suitable enlargement.  
Recall that we have chosen \(j\) so large that \(200Q \subset \tilde{Q}_0\).  
Set \(B_0 := 49B_Q\) and define the regularised measure \(\nu := \omega|_{3B_0} * \varphi_t\) with \(t \approx \ell_0^{N}\ell(Q)\).

\substep{Step 1: From \(\beta_{2,\omega}^n\) to \(\beta_{2,\nu}^n\):}
 Using Lemma~\ref{lem:beta-inequality} with \(r \in [t, 12\ell(Q)]\) we obtain
\[
\begin{aligned}
\int_{4B_Q} \int_t^{12\ell(Q)} \beta_{2,\omega}^n(x,r)^2 \,\Theta_\nu^n(x,r)\,\frac{dr}{r}\,d\nu(x)
&\lesssim \int_{4B_Q} \int_t^{12\ell(Q)} \beta_{2,\nu}^n(x,2r)^2 \,\Theta_\nu^n(x,2r)\,\frac{dr}{r}\,d\nu(x) \\
&\quad + \int_{4B_Q} \int_t^{12\ell(Q)} \Bigl(\frac{t}{r}\Bigr)^2 \Theta_\omega^n(x,3r)\,\frac{dr}{r}\,d\nu(x) \\
&\lesssim \int_{B_0} \int_0^{2\operatorname{rad}(B_0)} \beta_{2,\nu}^n(x,r)^2 \,\Theta_\nu^n(x,r)\,\frac{dr}{r}\,d\nu(x) \\
&\quad + \int_{4B_Q} |M_n(\nu|_{B_0})(x)|^2 \Bigl(\int_t^{12\ell(Q)} \Bigl(\frac{t}{r}\Bigr)^2\frac{dr}{r}\Bigr) d\nu(x).
\end{aligned}
\]
Since \(\int_t^\infty (t/r)^2 \, dr/r\) is bounded, the last term is controlled by \(\int_{B_0} |M_n(\nu|_{B_0})|^2 \, d\nu\).  
Combining with Lemma \ref{lower-bound1}, we obtain the estimate
\begin{equation}\begin{aligned}
\label{eq:step1}
   \delta_0^2 \ell(Q)^{2(s-n)}  & \sum_{P\in\mathcal{D}_{0,N}(Q)}\!\Theta_\omega^s(P)^2\omega(P)\\
   &
\lesssim\int_{B_0}\int_0^{2\operatorname{rad}(B_0)}\!\!\beta_{2,\nu}^n(x,r)^2\Theta_\nu^n(x,r)\frac{dr}{r}d\nu(x)
+ \int_{B_0}|M_n(\nu|_{B_0})|^2d\nu .
\end{aligned}
\end{equation}

\substep{Step 2: Estimating the right‑hand side via Theorem \ref{thm:Tolsa_main}:}
We now examine the terms on the right‑hand side of \eqref{eq:step1}.  
First, note that \(2B_0 \subset 100Q\) and that \(\nu\) is supported in \((3+t)B_0\).  
For the Riesz transform we write
\[
\mathcal{R}\nu = (\mathcal{R}\omega|_{3B_0})*\varphi_t = (\mathcal{R}\omega)*\varphi_t - (\mathcal{R}\omega|_{(3B_0)^c})*\varphi_t,
\]
hence
\[
\int_{2B_0}|\mathcal{R}\nu - m_{\nu,2B_0}(\mathcal{R}\nu)|^2d\nu \;\lesssim\; \int_{2B_0}|\mathcal{R}\nu|^2d\nu \le I_1 + I_2,
\]
where
\[
I_1 = \int_{2B_0}|(\mathcal{R}\omega)*\varphi_t|^2d\nu,\qquad 
I_2 = \int_{2B_0}|(\mathcal{R}\omega|_{(3B_0)^c})*\varphi_t|^2d\nu.
\]

\emph{Estimate of \(I_1\):} For any cube \(P\in\mathcal{D}_N(100Q)\) and any \(x\) such that $\operatorname{dist}(x,P)<t$, write \(\mathcal{R}\omega = \nabla g\) with \(g\) the Green function with the pole at infinity.  
By the usual mollifier estimates,
\[
\begin{aligned}
|(\mathcal{R}\omega)*\varphi_t(x)|^2
&= \Bigl|\int_{B(x,t)}\nabla g(y)\varphi_t(x-y)\,dy\Bigr|^2
\lesssim t^{n+1}\int_{B(x,t)}|\nabla g(y)|^2\varphi_t(x-y)^2\,dy\\
&\lesssim t^{-(n+1)}\int_{B(x,t)}|\nabla g(y)|^2\,dy
\lesssim t^{-(n+3)}\int_{B(x,2t)}|g(y)|^2\,dy,
\end{aligned}
\]  
where we also used the Caccioppoli inequality in the last estimate.
For \(y\in B(x,2t)\) the capacity density condition (Lemma~\ref{lem:CDC-Green}) yields
\[
|g(y)|  \lesssim_{c_{\mathrm{cap}}} \frac{\omega(16P)}{\ell(P)^{n-1}}.
\]
Since \(t\approx\ell(P)\), we obtain \(|(\mathcal{R}\omega)*\varphi_t(x)|^2 \lesssim_{c_{\mathrm{cap}}}  \Theta^n_\omega(16P)^2\).  
Summing over all \(P\in\mathcal{D}_N(100Q)\) gives
\[
I_1 \lesssim_{c_{\mathrm{cap}}}  \sum_{P\in\mathcal{D}_N(100Q)} \Theta^n_\omega(16P)^2 \,\nu(P)
      \lesssim \sum_{P\in\mathcal{D}_N(100Q)} \Theta^n_\omega(16P)^2 \,\omega(16P).
\]

\emph{Estimate of \(I_2\):} The function \(\mathcal{R}\omega|_{(3B_0)^c}\) is harmonic in \(3B_0\), therefore for \(x\in2B_0\),
\[
(\mathcal{R}\omega|_{(3B_0)^c})*\varphi_t(x) = \mathcal{R}\omega|_{(3B_0)^c}(x).
\]
Choose \(r=\ell(Q)\). By the mean value property for harmonic functions,
\begin{align*}
|\mathcal{R}\omega|_{(3B_0)^c}(x)| & \approx r^{-(n+1)}\Bigl|\int_{B(x,r)} \mathcal{R}\omega|_{(3B_0)^c}(y)\,dy\Bigr|\\
&\lesssim r^{-(n+1)}\int_{B(x,r)}\bigl(|\mathcal{R}\omega(y)|\,dy
+r^{-(n+1)}\int_{B(x,r)}|\mathcal{R}\omega|_{3B_0}(y)|\,dy.
\end{align*}
By Caccioppoli inequality, arguing as above,  the first intertal on the right hand side part is bounded by \(\Theta_\omega(\Lambda Q)\) (with \(\Lambda\) a fixed enlargement, e.g. \(\Lambda=100\)). The last integral is satisfies a similar estimate after interchanging integrals:
\[
r^{-(n+1)}\int_{B(x,r)}|\mathcal{R}\omega|_{3B_0}(y)|\,dy
\lesssim r^{-(n+1)}\int_{3B_0}\int_{4B_0}\frac{1}{|z|^n}\,dz\,d\omega(y)
\lesssim \Theta^n_\omega(100Q).
\]
Consequently \(|\mathcal{R}\omega|_{(3B_0)^c}(x)|^2 \lesssim_{c_{\mathrm{cap}}}  \Theta^n_\omega(100Q)^2\), and therefore
\[
I_2 \lesssim_{c_{\mathrm{cap}}}  \Theta^n_\omega(100Q)^2 \,\nu(2B_0) \lesssim \Theta^n_\omega(100Q)^2 \,\omega(100Q).
\]

Combining the estimates for \(I_1\) and \(I_2\) we obtain
\begin{equation}\label{eq:Riesz_nu}
\int_{2B_0}|\mathcal{R}\nu - m_{\nu,2B_0}(\mathcal{R}\nu)|^2d\nu
\;\lesssim_{c_{\mathrm{cap}}} \; \sum_{P\in\mathcal{D}_{N}(100Q)}\Theta^n_\omega(16P)^2\,\omega(16P)
\;+\; \Theta^n_\omega(100Q)^2\,\omega(100Q).
\end{equation}

\emph{Estimate of \(\mathcal{P}_\nu(B_0)\) and \(\theta_\nu^{n,*}\):}

For any \(x\in P\) with \(P\in\mathcal{D}_N(100Q)\) and \(s\ll t\),
\begin{align*}
\nu(B(x,s)) & = \int_{B(x,s)}\int_{3B_0}\varphi_t(y-z)\,d\omega(z)\,dy
\lesssim t^{-(n+1)}\int_{B(x,s)}\omega(B(z,t))\,dz\\
&\lesssim (s/t)^{\,n+1}\,\omega(B(x,2t)).
\end{align*}
Hence
\[
\theta_\nu^{n,*}(x) = \limsup_{s\to0}\frac{\nu(B(x,s))}{s^n}
\lesssim \limsup_{s\to0}\frac{s}{t^{n+1}}\,\omega(B(x,2t)) = 0.
\]
Moreover, \(\operatorname{supp}\nu\subset(3+t)B_0\), so for every \(k\ge0\),
\[
\nu(2^k B_0) \le \omega(4B_0),
\]
and therefore
\[
\mathcal{P}_\nu(B_0)^2\,\nu(2B_0) \lesssim \Theta^n_\omega(4B_0)^2\,\nu(4B_0) \lesssim \Theta^n_\omega(200Q)^2\,\omega(200Q).
\]

\vspace{1mm}
We now apply Theorem \ref{thm:Tolsa_main} to the measure \(\nu\) and the ball \(B_0\).  
The theorem gives
\[
\begin{aligned}
\int_{B_0}\!\int_0^{2\operatorname{rad}(B_0)}&\!\beta_{2,\nu}^n(x,r)^2\Theta^n_\nu(x,r)\frac{dr}{r}d\nu(x)
 + \int_{B_0}|M_n(\nu|_{B_0})|^2d\nu\\
&\lesssim \int_{2B_0}|\mathcal{R}\nu - m_{\nu,2B_0}(\mathcal{R}\nu)|^2d\nu+ \mathcal{P}_\nu(B_0)^2\nu(2B_0)
+ \int_{2B_0}\theta_\nu^{n,*}(x)^2d\nu(x).
\end{aligned}
\]
Inserting the estimates from Steps 1 and 2, and using that \(\theta_\nu^{n,*}=0\), we obtain
\[
\begin{aligned}
\delta_0^2 \ell(Q)^{2(s-n)} \!\!\!\sum_{P\in\mathcal{D}_{0,N}(Q)}\!\!\Theta_\omega^s(P)^2\omega(P)&\lesssim_{c_{\mathrm{cap}}} \; \Theta^n_\omega(200Q)^2\,\omega(200Q)
\,+ \!\!\sum_{P\in\mathcal{D}_{N}(100Q)}\!\!\Theta^n_\omega(16P)^2\,\omega(16P)\\
&\lesssim \; \ell(Q)^{2(s-n)}\Theta^s_\omega(200Q)^2\,\omega(200Q)\\
&+(\ell_0^N\ell(Q))^{2(s-n)} \sum_{P\in\mathcal{D}_{N}(100Q)}\Theta^s_\omega(16P)^2\,\omega(16P).
\end{aligned}
\]

If \(s=n\) the factor \(\ell(Q)^{s-n}\) disappears; otherwise we may absorb it by choosing \(N\) such that \(\ell_0^{-N(n-s)}\) is a not too large constant (for instance \(N \leq \frac1{n-s}\)).  
In any case we finally have
\begin{equation}\label{eq:local_estimate}
\delta_0^2 \sum_{P\in\mathcal{D}_{0,N}(Q)}\Theta_\omega^s(P)^2\omega(P)
\;\lesssim_{c_{\mathrm{cap}}} \; \Theta_\omega^s(200Q)^2\,\omega(200Q)
\;+\; \sum_{P\in\mathcal{D}_{N}(100Q)}\Theta_\omega^s(16P)^2\,\omega(16P).
\end{equation}

\substep{Step 3: Summation over all \(Q\in\mathcal{D}_j(Q_0)\)}
To obtain \eqref{eq:key_estimate} we sum \eqref{eq:local_estimate} over all \(Q\in\mathcal{D}_j(Q_0)\).  
The left‑hand side becomes \(\delta_0\, l_j(\omega)\).  
For the right‑hand side we split the sum according to the position of \(200Q\):
\[
\mathcal{K}_1 := \{ Q\in\mathcal{D}_j(Q_0) : 200Q \subset Q_0 \},\qquad
\mathcal{K}_2 := \{ Q\in\mathcal{D}_j(Q_0) : 200Q \cap (Q_0)^c \neq \varnothing \}.
\]

\begin{itemize}
\item For \(Q\in\mathcal{K}_1\), the term \(\Theta_\omega^s(200Q)^2\,\omega(200Q)\) is bounded by a finite  combination (with bounded overlap) of cubes \(P\in\mathcal{D}_j(Q_0)\).  
  More precisely, because \(200Q\) is covered by boundedly many cubes of the same generation, we have
  \[
  \sum_{Q\in\mathcal{K}_1} \Theta_\omega^s(200Q)^2\,\omega(200Q)
  \;\lesssim\; \sum_{P\in\mathcal{D}_j(Q_0)} \Theta_\omega^s(P)^2\,\omega(P).
  \]
\item For \(Q\in\mathcal{K}_2\), every \(x\in Q\) satisfies \(\operatorname{dist}(x,Q_0^c)\lesssim \ell(Q)\lesssim \ell_0^{j}\ell(Q_0)\).  
  Using the thin boundary property of the lattice (and the fact that \(\omega\) and \(\mu\) are comparable on these cubes by the stopping condition \eqref{eq:ratio_bound}), we estimate
  \[
  \begin{aligned}
  \sum_{Q\in\mathcal{K}_2} \Theta_\omega^s(200Q)^2\,\omega(200Q)
  &\lesssim \Bigl(\frac{\omega(Q_0)}{\mu(Q_0)}\Bigr)^3
     \sum_{Q\in\mathcal{K}_2} \sum_{P\subset 200Q} \Theta_\mu^s(P)^2\,\mu(P) \\
  &\lesssim_{C_0} \Bigl(\frac{\omega(Q_0)}{\mu(Q_0)}\Bigr)^3
     \sum_{Q\in\mathcal{K}_2} \mu(200Q)
   \lesssim_{C_0} \Bigl(\frac{\omega(Q_0)}{\mu(Q_0)}\Bigr)^3
     \sum_{Q\in\mathcal{K}_2} \mu(Q) \\
  &\le \Bigl(\frac{\omega(Q_0)}{\mu(Q_0)}\Bigr)^3
     \mu\bigl(\{x\in Q_0 : \operatorname{dist}(x,Q_0^c)\lesssim \ell_0^{j}\ell(Q_0)\}\bigr) \\
  &\lesssim_{C_0} \ell_0^{j\eta}\,\Bigl(\frac{\omega(Q_0)}{\mu(Q_0)}\Bigr)^3 \mu(Q_0)
   \approx_{C_0} \ell_0^{j\eta}\,\Theta^s_\omega(Q_0)^2\,\omega(Q_0).
  \end{aligned}
  \]
  In the last inequality we used the Thin boundary property of the David–Mattila lattice in Lemma \ref{dyadic-cube}.
\end{itemize}

The second term on the right‑hand side of \eqref{eq:local_estimate} is handled similarly:
\[
\begin{aligned}
\sum_{Q\in\mathcal{D}_j(Q_0)}\sum_{P\in\mathcal{D}_{N}(100Q)}\Theta_\omega^s(16P)^2\,\omega(16P)
&\lesssim \sum_{Q\in\mathcal{D}_{j}(Q_0)} \sum_{P\in\mathcal{D}_{N}(200Q)}\Theta_\omega^s(P)^2\,\omega(P) \\
&\lesssim \sum_{P\in\mathcal{D}_{j+N}(Q_0)} \Theta_\omega^s(P)^2\,\omega(P)
   \;+\; C^* \ell_0^{j\eta}\,\Theta^s_\omega(Q_0)^2\,\omega(Q_0).
\end{aligned}
\]

Putting everything together we arrive exactly at the estimate \eqref{eq:key_estimate}.  
This completes the proof of the lemma.
\end{proof}

\subsection{Averaging argument and choice of indices}\label{11.57}
Next, fix integers \(j \ge 1\) and \(m \ge 1\). For each $i\in \mathbb Z$ such that $j+i\,m\geq 0$, define the average
\[
a_i := \frac{1}{m}\sum_{t=1}^{m} \frac{1}{l_{\,j+t+im}(\mu)}.
\]
Recall from the Ahlfors regularity property \eqref{ADR} that every \(l_p(\mu)\) satisfies
\[
C_0^{-2}N\mu (Q_0)\lesssim \sum_{P\in\mathcal{D}_{p,p+N}(Q_0)}\Theta_\mu^s(P)^2\mu(P)=l_p(\mu)\lesssim C_0^{2}N\mu (Q_0).
\]
Consequently the numbers \(a_i\) are bounded.

We claim that there exists some $k$ with $|k|\lesssim C_0$ such that
\begin{equation}\label{eq:akk*}
2a_{k-1} > a_k \quad\text{and}\quad a_k \le 2a_{k+1}.
\end{equation}
Indeed, suppose that  this does not hold for $k=0$. In case that
$a_0>2a_1$, 
let \(k \ge 1\) be the smallest index such that \(a_k \le 2a_{k+1}\). Such a \(k\) exists because otherwise we would have \(a_i > 2a_{i+1}\) for all \(i>0\), which would imply \(a_i /a_0\le 2^{-i}\) and force \(a_i\) to become arbitrarily small, contradicting the lower bound \(a_i/a_0 \ge C_0^{-4}\). Moreover, a crude estimate using the bounds on \(a_i\) gives \(k \lesssim C_0\).  
From the minimality of \(k\) we obtain
\[
a_{k-1} > 2a_k \quad\text{and}\quad a_k \le 2a_{k+1},
\]
and so \eqref{eq:akk*} holds.
In case that $2a_{-1}<a_0$, we let let \(k \le -1\) be the largest index such that \(2a_{k-1} \ge a_{k}\) and we argue analogously.

Now let $K_1\in \{j+t+(k-1)m:t = 1,\dots,m \bigr\}$ and $K_2 \in \{j+t+(k+1)m:t = 1,\dots,m \bigr\}$ be such that
\[
l_{K_1}(\mu) = \min\bigl\{ l_{\,j+t+(k-1)m}(\mu) : t = 1,\dots,m \bigr\}\]
and
\[l_{K_2}(\mu) = \max\bigl\{ l_{\,j+t+(k+1)m}(\mu) : t = 1,\dots,m \bigr\}.
\]
Define three families of indices:
\[
\begin{aligned}
\mathcal{J}   &:= \{\, j+t+km \mid 8l_i(\mu) \ge l_{K_1}(\mu) \text{ and } 8l_i(\mu) \ge l_{K_2}(\mu) \,\},\\
\mathcal{J}_1 &:= \{\, j+t+km \mid 8l_i(\mu) < l_{K_1}(\mu) \,\},\\
\mathcal{J}_2 &:= \{\, j+t+km \mid 8l_i(\mu) < l_{K_2}(\mu) \,\}.
\end{aligned}
\]

From \(2a_{k-1} > a_k\) we obtain
\[
\frac{2}{l_{K_1}(\mu)} \ge 2a_{k-1} > a_k
= \frac{1}{m}\sum_{i=j+1+km}^{j+(k+1)m} \frac{1}{l_i(\mu)}
\ge \frac{1}{m}\sum_{i\in\mathcal{J}_1} \frac{1}{l_i(\mu)}
\ge \frac{1}{m}\cdot \frac{|\mathcal{J}_1|}{\max_{i\in\mathcal{J}_1} l_i(\mu)}.
\]
For \(i\in\mathcal{J}_1\) we have \(8l_i(\mu) < l_{K_1}(\mu)\), hence \(1/l_i(\mu) > 8/l_{K_1}(\mu)\).  
Thus the last term is at least \( \frac{8}{m}\frac{|\mathcal{J}_1|}{l_{K_1}(\mu)}\).  
Consequently
\[
\frac{1}{l_{K_1}(\mu)} > \frac{4}{m}\frac{|\mathcal{J}_1|}{l_{K_1}(\mu)} \quad\Longrightarrow\quad |\mathcal{J}_1| < \frac{m}{4}.
\]
A similar argument (using the inequality \(a_k \le 2a_{k+1}\)) gives \(|\mathcal{J}_2| \le \frac{m}{4}\).  
Therefore
\[
|\mathcal{J}| = m - |\mathcal{J}_1| - |\mathcal{J}_2| \ge m - \frac{m}{4} - \frac{m}{4} =  \frac{m}{2}.
\]

Now choose \(m \approx N\) (where \(N\) is as in Lemma~\ref{lemma:key_estimate}) and apply the key estimate \eqref{eq:key_estimate}.  
Summing over all indices in \(\mathcal{J}\) we obtain
\begin{equation}\label{eq:key-estimaten1}
\delta_0^2 \sum_{i\in\mathcal{J}} l_i(\omega) \;\lesssim_{c_{\mathrm{cap}}} \; l_{K_1}(\omega) + l_{K_2}(\omega) + C^* \ell_0^{j\eta} N \,\Theta_\omega^s(Q_0)^2\,\omega(Q_0). 
\end{equation}

Without loss of generality assume \(l_{K_1}(\omega) \ge l_{K_2}(\omega)\).  
From the definition of \(l_{K_1}(\omega)\) we have
\[
l_{K_1}(\omega) = \sum_{P\in\mathcal{D}_{K_1,K_1+N}(Q_0)} \Theta_\omega^s(P)^2 \,\omega(P)
\;\approx_{C_0}\; N\,\Theta_\omega^s(Q_0)^2\,\omega(Q_0). 
\]
Because \(j\) can be taken arbitrarily large, we can ensure that the error term in 
\eqref{eq:key-estimaten1} is negligible compared to \(l_{K_1}(\omega)\).  
Indeed, choose \(j_0 = j_0(C_0)\) so large that
\[
C^* 2^{-j\eta} N \,\Theta_\omega^s(Q_0)^2\,\omega(Q_0) \le \frac12\,l_{K_1}(\omega) \quad\text{for all }j\ge j_0.
\]
Then \eqref{eq:key-estimaten1} implies the existence of some index \(K_0 \in \mathcal{J}\) such that
\[
l_{K_0}(\omega) \;\le\; \frac{C(n,c_{\mathrm{cap}})}{N\delta_0^2}\,l_{K_1}(\omega). 
\]

On the other hand, by construction of \(\mathcal{J}\) we have \(8l_{K_0}(\mu) \ge l_{K_1}(\mu)\) and \(8l_{K_0}(\mu) \ge l_{K_2}(\mu)\); in particular
\[
l_{K_0}(\mu) \ge \frac{1}{8}\,l_{K_1}(\mu). 
\]

Now express \(l_{K_0}(\omega)\) in terms of \(\mu\).  Using the definition and the fact that on the cubes under consideration the ratio \(\omega(P)/\mu(P)\) is controlled by the stopping condition \eqref{eq:ratio_bound}, we obtain
\[
\begin{aligned}
l_{K_0}(\omega)
&= \sum_{P\in\mathcal{D}_{K_0,K_0+N}(Q_0)} \Theta_\omega^s(P)^2 \,\omega(P) \\
&= \sum_{P\in\mathcal{D}_{K_0,K_0+N}(Q_0)} \Bigl(\frac{\omega(P)}{\mu(P)}\Bigr)^3 \Theta_\mu^s(P)^2 \,\mu(P) \\
&\ge 2^{-24} \Bigl(\frac{\omega(Q_0)}{\mu(Q_0)}\Bigr)^3 \sum_{P\in\mathcal{D}_{K_0,K_0+N}(Q_0)} \Theta_\mu^s(P)^2 \,\mu(P) \\
&= 2^{-24} \Bigl(\frac{\omega(Q_0)}{\mu(Q_0)}\Bigr)^3 l_{K_0}(\mu) \ge2^{-27} \Bigl(\frac{\omega(Q_0)}{\mu(Q_0)}\Bigr)^3 l_{K_1}(\mu)\ge\; 2^{-51}\,l_{K_1}(\omega). 
\end{aligned}
\]

If we choose \(N \gtrsim_{c_{\mathrm{cap}}}  \frac{1}{\delta_0^2}\), we obtain a contradiction. 
Recall that in Lemma \ref{lemma:key_estimate} we asked $N\leq \frac1{n-s}$. The two conditions are compatible because
we assume that
$n-s\leq c\,\delta_0^2$, for a small $c=c(c_{\mathrm{cap}}) >0$.

\subsection{Choice of \(m_0\) and conclusion}
Finally, we choose \(m_0\) sufficiently large, e.g. \(m_0 \approx j_0 + N + C_0^2\).  
With this choice, for every cube \(Q \in \mathcal{D}_{j_0+2N+C_0^2}(Q_0)\) and every \(P \in \mathcal{D}_{0,N}(Q)\), the condition  
\[
\delta_0\ell(P) \ge \ell_0^{-m_0}\ell(Q_0)
\]  
required in Lemma \ref{lower-bound} is automatically satisfied.  
This completes the proof of Proposition \ref{prop:stopping}.
\end{proof}

\section{Proof of Theorem \ref{thm:main_higher_dim}}\label{11.28}

In this section, we complete the proof of Theorem \ref{thm:main_higher_dim}. With Proposition \ref{prop:stopping} already established, the remainder of the proof relies on a standard argument; see, for instance, \cite{T24, A20}. However, for the convenience of the reader and to keep the paper self-contained, we provide the  details of the proof below. We begin by deducing the following decay estimate from Proposition \ref{prop:stopping}.

\begin{lemma}\label{lem:decay}
    Under the assumptions of Proposition \ref{prop:stopping}, there exists $\gamma\in (0,1)$ depending on $m_0$ and $s$ such that for all $Q\in \mathcal{D}_\mu, $
    \begin{equation}
        \sum_{P\in \mathcal{D}_{m_0}(Q)}\omega(P)^{1/2}\mu(P)^{1/2}\le \gamma \,\omega(Q)^{1/2}\mu(Q)^{1/2}.
    \end{equation}
\end{lemma}
\begin{proof}
    First,  Proposition \ref{prop:stopping} implies that we can select  $P_0 \in \mathcal{D}_{0,m_0}(Q)$ satisfies either
    \begin{equation}\label{shrink_cube}
        \frac{\omega (P_0)}{\mu(P_0)} \le 16^{-1}\, \frac{\omega (Q)}{\mu (Q)},
    \end{equation}
    or
    \begin{equation}\label{blowup_cube}
       \frac{\omega (P_0)}{\mu(P_0)}  \ge 16 \,\frac{\omega (Q)}{\mu (Q)} .
    \end{equation}
    
Assume first that \eqref{shrink_cube} holds. We split the sum into two parts: the terms contained in \(P_0\), and the rest. By the Cauchy-Schwarz inequality applied to the cubes in $Q \setminus P_0$, we have
    \[
     \sum_{P\in \mathcal{D}_{m_0}(Q)\setminus \mathcal{D}_{\mu}(P_0)}
    \omega(P)^{1/2} \mu(P)^{1/2} \leqslant \omega(Q \setminus P_0)^{1/2} \mu(Q \setminus P_0)^{1/2} \le \omega(Q)^{1/2} (\mu(Q) - \mu(P_0))^{1/2}.
    \]
    For the term $P_0$, \eqref{shrink_cube}  yields
    \[
    \omega(P_0)^{1/2} \mu(P_0)^{1/2} \le  \frac14\, \omega(Q)^{1/2} \mu(Q)^{1/2} \frac{\mu(P_0)}{\mu(Q)},
    \]
   and  using the elementary inequality $(1-x)^{1/2} \le 1 - x/2$ for $0\le x\le 1$, we verify 
    \[
    \begin{aligned}
     \sum_{P\in \mathcal{D}_{m_0}(Q)}\!\!\omega(P)^{1/2}\mu(P)^{1/2}&= \!\! \sum_{P\in \mathcal{D}_{m_0}(Q)\cap \mathcal{D}_{\mu}(P_0)} \!\!\omega(P)^{1/2} \mu(P)^{1/2}+\!\!
      \sum_{P\in \mathcal{D}_{m_0}(Q)\setminus \mathcal{D}_{\mu}(P_0)}\!\!\omega(P)^{1/2} \mu(P)^{1/2}\\
     &\le\omega(P_0)^{1/2} \mu(P_0)^{1/2} + \sum_{P\in \mathcal{D}_{m_0}(Q)\setminus \mathcal{D}_{\mu}(P_0)} \omega(P)^{1/2} \mu(P)^{1/2}\\
     & \le \omega(Q)^{1/2} \mu(Q)^{1/2} \left[ \left(1 - \frac{\mu(P_0)}{\mu(Q)}\right)^{1/2} + \frac{1}{4}\frac{\mu(P_0)}{\mu(Q)} \right] \\
     &\le \left(1 - \frac{1}{4} \frac{\mu(P_0)}{\mu(Q)} \right) \omega(Q)^{1/2} \mu(Q)^{1/2}.
    \end{aligned}
    \]
     Since
     $$ \frac{\mu(P_0)}{\mu(Q)}\approx_{C_0}\frac{\ell(P_0)^s}{\ell(Q)^s}\ge\ell_0^{m_0 s}, $$we obtain the factor $\gamma = 1 - c \ell_0^{m_0 s}$.
     
    Suppose now that \eqref{blowup_cube} holds. The arguments are quite similar to the previous ones, interchanging the roles of $\mu$ and $\omega$. Indeed, By Cauchy-Schwarz,
\[
    \sum_{P\in \mathcal{D}_{m_0}(Q)\setminus \mathcal{D}_{\mu}(P_0)} \omega(P)^{1/2} \mu(P)^{1/2} \le (\omega(Q) - \omega(P_0))^{1/2} \mu(Q)^{1/2}.
\]
Also, we have
\begin{align*}
    \omega(P_0)^{1/2} \mu(P_0)^{1/2}  \omega(P_0) \frac{\mu(P_0)^{1/2}}{\omega(P_0)^{1/2}} \le \frac{1}{4} \omega(P_0)(\frac{\mu(Q)}{\omega(Q)})^{1/2} = \frac{1}{4} \frac{\omega(P_0)}{\omega(Q)} \omega(Q)^{1/2} \mu(Q)^{1/2}.
\end{align*}
    From two previous estimates and using again the inequality $(1 - x)^{1/2} \le 1 - \frac{1}{2}x$ for $0 \le x \le 1$, we obtain
\begin{align*}
    \sum_{P \in \mathcal{D}_{m_0}(Q)} \omega(P)^{1/2}\mu(P)^{1/2} &\le (\omega(Q) - \omega(P_0))^{1/2}\mu(Q)^{1/2} + \frac{1}{4}\frac{\omega(P_0)}{\omega(Q)} \omega(Q)^{1/2}\mu(Q)^{1/2} \\
    &= \omega(Q)^{1/2}\mu(Q)^{1/2} \left( \left( 1 - \frac{\omega(P_0)}{\omega(Q)} \right)^{1/2} + \frac{1}{4}\frac{\omega(P_0)}{\omega(Q)} \right) \\
    &\le \omega(Q)^{1/2}\mu(Q)^{1/2} \left( 1 - \frac{1}{4} \frac{\omega(P_0)}{\omega(Q)} \right).
\end{align*}
 Observe now that
\[
    \frac{\omega(P_0)}{\omega(Q)}\ge 16 \frac{\mu(P_0)}{\mu(Q)}\gtrsim \ell_0^{m_0s}.
\]
Taking as before, $\gamma = 1 - c \ell_0^{m_0 s}$, we get
\[
    \sum_{P \in \mathcal{D}_{m_0}(Q)} \omega(P)^{1/2} \mu(P)^{1/2} \le \gamma \omega(Q)^{1/2} \mu(Q)^{1/2}.
\]
Finally, we completes the proof of the lemma. 
\end{proof}

\begin{proof}[Proof to Theorem \ref{thm:main_higher_dim}]
We introduce now a dyadic Hausdorff content for subsets of $E := \operatorname{supp}(\mu)$. For $F \subset E$ and $t, \varepsilon > 0$, we denote
\begin{equation}\label{eq:12.1}
    \mathcal{H}_\varepsilon^s(F) = \inf \Big\{ \sum_i \operatorname{diam}(A_i)^s : F \subset \bigcup_i A_i, \operatorname{diam}(A_i) \le \varepsilon \Big\}.
\end{equation}
and
\begin{equation}\label{eq:12.2}
    \mathcal{H}_\varepsilon^{\mathcal{D},s}(F) = \inf \Big\{ \sum_i \ell(Q_i)^s : Q_i \in \mathcal{D}, F \subset \bigcup_i Q_i, \ell(Q_i) \le \varepsilon \Big\},
\end{equation}
By the properties of $\mathcal{D}_\mu$, it is immediate to check that $\mathcal{H}_\varepsilon^t(F) \approx \mathcal{H}_\varepsilon^{\mathcal{D}_\mu,t}(F)$ with the implicit constant depending on the parameters in the definition of $\mathcal{D}_\mu$, which depend only on $n$.

We will show that for every cube $R_0 \in \mathcal{D}_\mu$, $\dim_{\mathcal{H}}(\omega|_{R_0}) \le t$ for some $t < s$ depending on $\gamma$ in Lemma \ref{lem:decay}, which suffices to prove Theorem \ref{thm:main_higher_dim}. To this end, we will prove that, for every $\tau > 0$, there exists a subset $E_\tau \subset E \cap R_0$ satisfying
    \begin{equation}\label{eq:12.19}
    \mathcal{H}_\infty^t(E_\tau) \le \tau \quad \text{and} \quad \omega(E \cap R_0 \setminus E_\tau) \le \tau.
\end{equation}

It is immediate to check that this implies that $\dim_{\mathcal{H}}(\omega|_{R_0}) \le t$.

For $m_0$ as in Lemma \ref{lem:decay}, for every $k \ge 1$ we have
\begin{align*}
    \sum_{Q \in \mathcal{D}_{ k m_0}(R_0)} \omega(Q)^{1/2}\mu(Q)^{1/2} &= \sum_{P \in \mathcal{D}_{ (k-1) m_0}(R_0)} \sum_{Q \in \mathcal{D}_{ m_0}(P)} \omega(Q)^{1/2}\mu(Q)^{1/2} \\
    &\le \gamma \sum_{P \in \mathcal{D}_{ (k-1) m_0}(R_0)} \omega(P)^{1/2}\mu(P)^{1/2}.
\end{align*}

Iterating, we deduce that
\begin{equation}\label{eq:12.20}
    \sum_{Q \in \mathcal{D}_{ k m_0}(R_0)} \omega(Q)^{1/2}\mu(Q)^{1/2} \le \gamma^k \omega(R_0)^{1/2}\mu(R_0)^{1/2} \quad \text{for all } k \ge 1.
\end{equation}

For any fixed $k \ge 1$, denote $\delta_k = \ell_0^{k m_0}$, so that $\ell(Q) = \delta_k \ell(R_0)$ for $Q \in \mathcal{D}_{ k m_0}(R_0)$. For some $t' \in (0, s)$ to be fixed below, consider the families
\[
    S_k^1 = \Big\{ Q \in \mathcal{D}_{ km_0}(R_0) : \omega(Q) \ge \Big( \frac{\ell(Q)}{\ell(R_0)} \Big)^{t'} \omega(R_0) \Big\}, \qquad S_k^2 = \mathcal{D}_{ km_0}(R_0) \setminus S_k^1.
\]

We have
\[
    \ell(Q)^{t'} \le \frac{\omega(Q)}{\omega(R_0)} \ell(R_0)^{t'} \quad \text{for each } Q \in S_k^1,
\]
and thus
\begin{equation}\label{eq:12.21}
    \sum_{Q \in S_k^1} \ell(Q)^{t'} \le \sum_{Q \in S_k^1} \frac{\omega(Q)}{\omega(R_0)} \ell(R_0)^{t'} \le \ell(R_0)^{t'}.
\end{equation}

On the other hand, the cubes $Q \in S_k^2$ satisfy
\[
    \omega(Q) < \Big( \frac{\ell(Q)}{\ell(R_0)} \Big)^{t'} \omega(R_0) = \delta_k^{t'-s} \Big( \frac{\ell(Q)}{\ell(R_0)} \Big)^s \omega(R_0),
\]
and so, by \eqref{eq:12.20},
\begin{align*}
    \sum_{Q \in S_k^2} \omega(Q) &\le \delta_k^{(t'-s)/2} \sum_{Q \in \mathcal{D}_{ km_0}(R_0)} \omega(Q)^{1/2} \Big( \frac{\ell(Q)}{\ell(R_0)} \Big)^{s/2} \omega(R_0)^{1/2} \\
    &\lesssim_{C_0} \delta_k^{(t'-s)/2} \sum_{Q \in \mathcal{D}_{ km_0}(R_0)} \omega(Q)^{1/2} \mu(Q)^{1/2} \Big( \frac{\omega(R_0)}{\mu(R_0)} \Big)^{1/2} \le \delta_k^{(t'-s)/2} \gamma^k \omega(R_0).
\end{align*}

Recalling that $\delta_k = \ell_0^{-km_0}$, for $t'$ close enough to $s$ we have
\[
    \delta_k^{(t'-s)/2} \gamma^k = \ell_0^{-km_0(t'-s)/2} \gamma^k \le \gamma^{k/2},
\]
and thus
\begin{equation}\label{eq:12.22}
    \sum_{Q \in S_k^2} \omega(Q) \lesssim \gamma^{k/2} \omega(R_0).
\end{equation}

Let $t = (t'+s)/2$ and denote
\[
    E_\tau = \bigcup_{Q \in S_k^1} Q.
\]

Since the cubes $Q$ in the union above satisfy $\ell(Q) \le \delta_k \ell(R_0)$, by \eqref{eq:12.21}, we have
\[
    \mathcal{H}_\infty^t(E_\tau) \lesssim (\delta_k \ell(R_0))^{t-t'} \mathcal{H}_{\delta_k \ell(R_0)}^{\mathcal{D}_\mu, t'} \Big( \bigcup_{Q \in S_k^1} Q \Big) \lesssim \delta_k^{t-t'} \ell(R_0)^t = \delta_k^{s-t} \ell(R_0)^t.
\]

On the other hand, notice that
\[
    E \cap R_0 \subset E_\tau \cup \bigcup_{Q \in S_k^2} Q.
\]

Then, by \eqref{eq:12.22}, $\omega(E \cap R_0 \setminus E_\tau) \lesssim \gamma^{k/2} \omega(R_0)$. Hence, for $k$ large enough \eqref{eq:12.19} follows. 
\end{proof}

\section{Proof of Theorem \ref{thm:main_planar} : The Planar Case}
In this section, we present the proof of Theorem \ref{thm:main_planar}. As outlined in Section \ref{main-idea}, the geometric hypothesis $\beta_\infty(x,r) + \beta_{\operatorname{hole}}(x,r) \ge \delta_0$ naturally divides our analysis into two regimes. We organize the section as follows: First, in Lemma \ref{6.46}, we use a compactness argument for the regime where $\beta_\infty$ is small, utilizing the fact that the blow-up limit is contained in a line. 
Next, we address the complementary regime by exploiting the specific planar property that $\beta_\infty$ is controlled by $\beta_2$. Proposition \ref{prop:stopping1} synthesizes these two regimes to secure the existence of a stopping cube with a definitive density jump. This allows us to conclude the proof by applying the standard dimension drop machinery from Section \ref{11.28}.

\begin{lemma}\label{6.46}
Let $s \in [2/3, 1]$, $\delta \in( 0,1)$, and $M>1$. Then, there exist constants $\tau,\kappa \in(0,1)$, both depending only on  $C_0$ and $M$, and independent of $s$ and $\delta$, such that the following holds: Suppose $E$ is an $(s,C_0)$-Ahlfors regular compact set. Let $x \in E$ and $0 < r < \operatorname{diam}(E)$. If $\beta_{\infty}(x,10r) < \tau \delta$ and $\beta_{\operatorname{hole}}(x,r) > \delta$, then there exist balls $B_1 \subset B_2 \subset B(x,r)$ centered on $E$ with $r(B_2) \geq\kappa \delta r$ and  $r(B_1) \geq\kappa \delta r$ such that$$\frac{\omega(B_1)}{r(B_1)^s} > M\frac{\omega(B_2)}{r(B_2)^s}.$$\end{lemma}
\begin{proof}
The proof relies on a compactness argument in the Attouch-Wets topology, which is a local variant of the Hausdorff metric topology for the convergence of sets. For $r > 0$ and non-empty sets $E, F \subset \mathbb{R}^2$, we define
\[
d_r(E, F) = \max\Big( \sup_{x \in E \cap \bar{B}(0,r)} \operatorname{dist}(x, F), \sup_{x \in F \cap \bar{B}(0,r)} \operatorname{dist}(x, E) \Big).
\]
A sequence of sets $E_k \subset \mathbb{R}^2$ is said to converge to a set $F \subset \mathbb{R}^2$ in the Attouch-Wets topology if $d_r(E_k, F) \to 0$ as $k \to \infty$ for every $r > 0$. Similar to the classical Hausdorff convergence for compact subsets, any sequence of closed subsets $E_k \subset \mathbb{R}^2$ intersecting a fixed ball admits a subsequence that converges in the Attouch-Wets topology to some closed set in $\mathbb{R}^2$ (see \cite{BL15} for a recent exposition).

Suppose, for the sake of contradiction, that the lemma fails. Then, for fixed constants $C_0, M > 1$, there exist sequences $\delta_k \in (0,1)$, $\tau_k \to 0$, $\kappa_k\to0$, $s_k \in [2/3, 1]$, and $(s_k, C_0)$-Ahlfors regular compact sets $E_k \subset \mathbb{R}^2$ with points $x_k \in E_k$ and radii $0 < r_k \leq \operatorname{diam}(E_k)$ such that 
\[
\beta_{\infty, E_k}(x_k, 10r_k) \leq \tau_k \delta_k \quad \text{and} \quad \beta_{\operatorname{hole}, E_k}(x_k, r_k) > \delta_k,
\]
and moreover, for all balls $B_1 \subset B_2 \subset B(x_k, r_k)$ centered on $E_k$ with $r(B_1), r(B_2) \geq\kappa_k \delta_k r_k$, we have
\begin{equation}\label{eq:contra_assumption}
\frac{\omega_k^{\infty}(B_1)}{r(B_1)^{s_k}} \le M \frac{\omega_k^{\infty}(B_2)}{r(B_2)^{s_k}},
\end{equation}
where $\omega_k$ is the harmonic measure for $\mathbb{R}^2 \setminus E_k$.

By translating and dilating $E_k$, we may assume that $x_k = 0$ and $r_k = 1$. The conditions $\beta_{\infty, E_k}(0, 10) \leq \tau_k \delta_k$ and $\beta_{\operatorname{hole}, E_k}(0, 1) > \delta_k$ imply the existence of a line $L_k$ such that $B(0,10) \cap E_k \subset \{x : \operatorname{dist}(x, L_k) \le 10\tau_k \delta_k\}$ and $\sup_{y \in B(0,1) \cap L_k} \operatorname{dist}(y, E_k) > \delta_k$. Consequently, there exists a point $z_k \in B(0,1) \cap L_k$ such that $\operatorname{dist}(z_k, E_k) \ge \delta_k / 2$. Let $y_k \in E_k \cap B(0,2)$ be a touching point realizing this distance, so that $\operatorname{dist}(z_k, E_k) = \operatorname{dist}(z_k, y_k)$.

We now introduce the magnified sets $\tilde{E}_k = \frac{E_k - y_k}{\delta_k/4}$. Under this scaling, the beta numbers satisfy
\[
\beta_{\infty, \tilde{E}_k}(0,8) = \beta_{\infty, E_k}(y_k, 2\delta_k) \le \frac{10}{2\delta_k} \beta_{\infty, E_k}(0, 10) \lesssim \tau_k.
\]
 By a suitable rotation, we can assume that the best approximating plane $\tilde L_k$ is the horizontal axis and that $B((-1,0),1/2)\cap \tilde E_k=\varnothing$.
 
 Let $\tilde{\omega}_k$ denote the corresponding harmonic measure for the re-scaled domain.
By conformal invariance, the inequality \eqref{eq:contra_assumption} is preserved for radii $\gtrsim \kappa_k$. Thus, for any balls $B(0, r_1) \subset B(0, r_2)\subset B(0,4)$ with $r_1, r_2 \gtrsim \kappa_k$, it holds that
\begin{equation}\label{eq:scaled_contra}
\frac{\tilde{\omega}_k^{\infty}(B(0,r_1))}{r_1^{s_k}} \le M\frac{\tilde{\omega}_k^{\infty}(B(0,r_2))}{r_2^{s_k}}.
\end{equation}

Passing to a subsequence, we may assume that $s_k \to s \in [2/3, 1]$ and that the sets $\tilde{E}_k$ converge in the Attouch-Wets topology to an $(s, C_0)$-Ahlfors regular set $\widetilde{E} \subset \mathbb{R}^2$. Let $\tilde{\omega}_k^\infty$ denote the harmonic measure with pole at infinity; it follows that $\tilde{\omega}_k^\infty$ converges weakly-* to $\widetilde{\omega}^\infty$, the harmonic measure for $\widetilde{\Omega} = \mathbb{R}^2 \setminus \widetilde{E}$ with pole at infinity.

Since $\beta_{\infty, \tilde{E}_k}(0, 8) \lesssim \tau_k \to 0$ as $k \to \infty$, we deduce that $\beta_{\infty, \widetilde{E}}(0, 8) = 0$. Because $0 \in \widetilde{E}$, this means $\widetilde{E} \cap B(0,8)$ is completely contained in a line, which we may assume to be the horizontal axis and also that $B((-1,0),1/2)\cap \widetilde E=\varnothing$. Furthermore, from \eqref{eq:scaled_contra} and the weak-* convergence of the measures, we fix $r_2 = 1/100$ and obtain for $0<r \leq 1/100$:
\[
\frac{\widetilde{\omega}^\infty(B(0, r))}{r^s} \lesssim M \,\widetilde{\omega}^\infty(B(0, 1/100)).
\]

Now, let $x_0 = (0, 1/200)$. Notice that $B(0,8)\setminus \widetilde{E} $ is a semi-uniform domain satisfying the CDC,\footnote{Since $s$ can be exactly $1$, the domain $B(0,8)\setminus \widetilde{E} $ might fail to be a uniform domain.} and then by
the change of pole formula proved by Azzam (see \cite{A21}), it holds that
\[
\widetilde{\omega}^{x_0}(B(0, r)) \lesssim \frac{\widetilde{\omega}^\infty(B(0, r))}{\widetilde{\omega}^\infty(B(0, 1/100))} \lesssim M r^s.
\]
We remark that an inspection of the arguments in \cite{A21} shows that the change of pole formula does not require the
whole domain $\R^2\setminus \widetilde E$ to be semi-uniform, and instead, it suffices $B(0,8)\setminus \widetilde E$ to be semi-uniform.

On the other hand, applying the techniques from \cite[Section 3]{T24}, we have the lower bound $\widetilde{\omega}^{x_0}(B(0, r)) \gtrsim r^{1/2}$. Combining the previous estimates gives
\[
r^{1/2} \lesssim M r^s \implies r^{1/2 - s} \lesssim M.
\]
Because $s \in [2/3, 1]$, the exponent satisfies $1/2 - s \le -1/6 < 0$. As $r \to 0$, the term $r^{1/2 - s}$ diverges to $+\infty$, which drastically contradicts the fact that it is bounded by the fixed constant $M$. 
 \end{proof}

With the compactness argument of Lemma \ref{6.46} concluded, we have  handled the geometric regime where $\beta_\infty$ is  small. For the complementary regime, as outlined at the beginning of this section, this is precisely the scenario where the planar relationship between $\beta_\infty$ and  $\beta_2$ comes into play (see Lemma \ref{beta_infty} below).

\begin{proposition}\label{prop:stopping1}

 Under the assumptions of Theorem \ref{thm:main_planar}, there exists $m_1$ depending on $\delta_0$ and $C_0$ 
 and $\gamma\in (0,1)$ depending on $m_1$ and $s$ such that for all $Q\in \mathcal{D}_\mu, $
    \begin{equation}\label{3.17}
        \sum_{P\in \mathcal{D}_{m_1}(Q)}\omega(P)^{1/2}\mu(P)^{1/2}\le \gamma \omega(Q)^{1/2}\mu(Q)^{1/2}.
    \end{equation}
\end{proposition}
\begin{proof}We proceed by a case analysis on the $L^\infty$-flatness of the descendants of $Q$. Let $m_0$ be the generation gap given by Proposition \ref{prop:stopping}, and let $\tau > 0$ be the constant from Lemma \ref{6.46}. We introduce a parameter $\delta \approx \delta_0$, which will be specified later.

\textbf{Case 1:} Suppose for all $P \in \mathcal{D}_{0,m_0}(Q)$, we have
\begin{equation}\label{non-flat beta infty}
  \beta_{\infty}(P) \ge \tau\delta.  
\end{equation}

As in the proof of Proposition \ref{prop:stopping}, our goal is to show that there exists a cube $P_0\in \mathcal{D}_{0,m_0}(Q) $ such that either
\[
\frac{\omega(P_0)}{\mu(P_0)}\;\ge\; 16\,\frac{\omega(Q)}{\mu(Q)}\qquad\text{or}\qquad 
\frac{\omega(P)}{\mu(P)}\;\le\; 16^{-1}\,\frac{\omega(Q)}{\mu(Q)}.
\]  
Suppose for the sake of contradiction that this is not the case. Then, for all $P\in \mathcal{D}_{0,m_0}(Q) $
we have
\begin{equation}\label{eq:ratio_bound1}
16^{-1}\,\frac{\omega(Q)}{\mu(Q)}\;<\;\frac{\omega(P)}{\mu(P)}\;<\;16\,\frac{\omega(Q)}{\mu(Q)}.
\end{equation}
 
\begin{claim}\label{2.54}
  Fix $R\in\mathcal{D}_{0,\,m_0-N}(Q)$ and under the assumptions of \eqref{non-flat beta infty} and \eqref{eq:ratio_bound1}, we have 
   \[
   \delta_0^2 \ell(R)^{2(s-1)}\, \!\!\sum_{P\in\mathcal{D}_{0,N}(R)}\!\!\Theta_\omega^s(P)^2\omega(P)
   \lesssim_{C_0} \int_{4B_R} \int_t^{12\ell(R)} \beta_{2,\omega}^1(x,r)^2 \,\Theta_\nu^1(x,2r)\,\frac{dr}{r}\,d\nu(x).
   \]
\end{claim}
\begin{proof}[Proof of Claim \ref{2.54}]

First, recalling the assumption \eqref{eq:ratio_bound1} and proceeding as in the proof of Lemma
\ref{lower-bound} and \eqref{eq:ratio_bound1}, we obtain 
\[
\begin{aligned}
\int_{4B_R} &\int_t^{12\ell(R)} \beta_{2,\omega}^1(x,r)^2 \,\Theta_\nu^1(x,2r)\,\frac{dr}{r}\,d\nu(x) \\
&\!\!\!\!\gtrsim_{C_0} \Theta_\omega^s(R)^2 \ell(R)^{2(s-1)}\frac{\omega(R)}{\mu(R)} \sum_{P\in\mathcal{D}_{0,N}(R)} \beta_{2,\mu}^s(P)^2 \ell(P)^s- C\epsilon^2 N \,\Theta_\omega^s(R)^2 \ell(R)^{2(s-1)} \omega(R).
\end{aligned}
\]

Note that \[\beta_{2,\mu}^s(P)^2\gtrsim_{C_0} \beta_{1,\mu}^s(P)^2\]
and combine with 
the inequality in Lemma \ref{beta_infty} gives the lower bound
\[
\sum_{P\in\mathcal{D}_{0,N}(R)} \beta_{2,\mu}^s(P)^2 \ell(P)^s \gtrsim_{C_0} \sum_{P\in\mathcal{D}_{0,N}(R)} \beta_{\infty}(P)^2 \ell(P)^s - \ell(R)^s.
\]
Substituting this into the previous estimate, we obtain
\[
\begin{aligned}
\int_{4B_R} \int_t^{10\ell(R)} \beta_{2,\omega}^1(x,r)^2 &\,\Theta_\nu^1(x,2r)\,\frac{dr}{r}\,d\nu(x)\\
&\gtrsim_{C_0} \Theta_\omega^s(R)^2 \ell(R)^{2(s-1)} \frac{\omega(R)}{\mu(R)}
\Bigl( \sum_{P\in\mathcal{D}_{0,N}(R)} \beta_{\infty}(P)^2 \ell(P)^s - \ell(R)^s \Bigr) \\
&\qquad - C\epsilon^2 N \,\Theta_\omega^s(R)^2 \ell(R)^{2(s-1)} \omega(R).
\end{aligned}
\]

Now, the  assumption \eqref{non-flat beta infty} implies $\beta_{\infty}(P)\ge \tau\delta \gtrsim \delta_0$ for every $P\in\mathcal{D}_{0,\,N}(R)$. Using the regularity property $\mu(P) \approx_{C_0} \ell(P)^s$, the first sum is bounded below by $\delta_0^2 \sum_{P\in\mathcal{D}_{0,N}(R)} \mu(P)$. Because the cubes $\mathcal{D}_{0,N}(R)$ form a partition of $R$ at the $N$-th generation, we have $\sum_{P\in \mathcal{D}_{0,N}(R)} \mu(P) \approx N \mu(R)$. Hence the dominant positive term is proportional to $\delta_0^2 N \,\Theta_\omega^s(R)^2 \ell(R)^{2(s-1)} \omega(R)$.

To absorb the negative terms, we choose $\epsilon$ so that $\epsilon = c \delta_0$ with a sufficiently small constant $c>0$; then the penalty term $C\epsilon^2 N$ is dominated by $\delta_0^2 N$. Additionally, we take $N$ large enough that $\delta_0^2 N \gtrsim_{C_0} 1$, which allows us to absorb the $\ell(Q)^s$ term. Under these choices, the lower bound simplifies to
\[
\int_{4B_R} \int_t^{12\ell(R)} \beta_{2,\omega}^1(x,r)^2 \,\Theta_\nu^1(x,2r)\,\frac{dr}{r}\,d\nu(x)
\gtrsim_{C_0} \delta_0^2 \,\Theta_\omega^s(R)^2 \ell(R)^{2(s-1)} \frac{\omega(R)}{\mu(R)} \sum_{P\in\mathcal{D}_{0,N}(R)} \mu(P).
\]

Finally, we use again \eqref{eq:ratio_bound} to rewrite the right‑hand side as
\[
\Theta_\omega^s(R)^2 \frac{\omega(R)}{\mu(R)}\sum_{P\in\mathcal{D}_{0,N}(R)} \mu(P)
\approx_{C_0} \sum_{P\in\mathcal{D}_{0,N}(R)} \Theta_\omega^s(P)^2 \omega(P).
\]
This establishes the desired inequality and completes the proof.
\end{proof}

Therefore, by invoking Claim \ref{2.54} in place of Lemma \ref{lower-bound1}, in Case 1 the desired estimate \eqref{3.17} follows directly from the arguments used in the proof of Proposition \ref{prop:stopping}, in combination with Lemma \ref{lem:decay}.

\textbf{Case 2:} Suppose conversely that there exists at least one cube $P_0 \in \mathcal{D}_{0,m_0}(Q)$
such that
$$\beta_{\infty}(P_0) < \tau\delta.$$
By the scaling properties of dyadic cubes (Lemma \ref{dyadic-cube}), there exists an associated ball $\tilde{B}_{P_0}$ concentric with $P_0$, with $\tilde{B}_{P_0}\cap E\subset P_0$, with radius comparable to the side length of $P_0$, and such that $\beta_{\infty}(\tilde{B}_{P_0}) \lesssim \beta_{\infty}(P_0)$. Recall from the hypothesis of Theorem \ref{thm:main_planar} that 
$\beta_{\infty}(\tilde{B}_{P_0})+ \beta_{\operatorname{hole}}(\tilde{B}_{P_0})\geq \delta_0$. Thus, by choosing $\delta$ sufficiently small relative to $\delta_0$, the upper bound on $\beta_{\infty}(\tilde{B}_{P_0})$ forces the topological term $\beta_{\operatorname{hole}}(\tilde{B}_{P_0})$ to be large. Specifically, we obtain$$\beta_{\infty}(\tilde{B}_{P_0}) \le \tau\delta \quad \text{and} \quad \beta_{\operatorname{hole}}(\tilde{B}_{P_0}/10) \ge \delta.$$
Applying Lemma \ref{6.46}, there exist balls $B_1 \subset B_2 \subset \tilde{B}_{P_0}$ with radii $r(B_i) \gtrsim \delta \ell(P_0)$ such that
\[
\frac{\omega(B_1)}{r(B_1)^s} > M\frac{\omega(B_2)}{r(B_2)^s}.
\]
We now transfer this estimate from balls to dyadic cubes. By the completeness and nesting properties of the dyadic grid, we can find dyadic cubes $Q_1$ and $Q_2$ such that $Q_1 \subset B_1 \subset B_2 \subset Q_2 \subset P_0$ with comparable sizes: $\ell(Q_1) \approx r(B_1)$ and $\ell(Q_2) \approx r(B_2)$. Due to the Scaling property of Lemma \ref{dyadic-cube} and the fact that $E$ is an $(s, C_0)$-Ahlfors regular compact set, $r(B_i)^s$ is  comparable to $\mu(Q_i)$ up to a constant depending  on $C_0$. Consequently, 
\[
\frac{\omega(Q_1)}{\mu(Q_1)} \ge c(C_0) M \frac{\omega(Q_2)}{\mu(Q_2)}.
\]
 Since Lemma \ref{6.46} allows us to choose the constant $M$ to be arbitrarily large depending only on $C_0$,  then this guarantees the existence of cubes $Q_1$ and $Q_2$ with $Q_1 \subset Q_2 \subset P_0$ satisfying the desired density jump:
\[
\frac{\omega(Q_1)}{\mu(Q_1)} \ge 16\frac{\omega(Q_2)}{\mu(Q_2)}.
\]
Furthermore, the side lengths satisfy $\ell(Q_1) \approx_{C_0} \ell(Q_2)$ and $\ell(Q_2) \gtrsim_{C_0} \delta \ell(P_0)$. In terms of dyadic generations, we may assume that $Q_1 \in \mathcal{D}_{N_1}(Q_2)$ and $Q_2 \in \mathcal{D}_{N_2}(P_0)$, where $N_1 \approx_{C_0} 1$ and $N_2 \lesssim |\log \delta|$.

With $Q_1$ and $Q_2$ at hand, we proceed by applying the argument from Lemma \ref{lem:decay} to the cube $Q_2$, which yields$$\sum_{P \in \mathcal{D}_{ N_1}(Q_2)} \omega(P)^{1/2} \mu(P)^{1/2} \le \gamma' \omega(Q_2)^{1/2}\mu(Q_2)^{1/2},$$where $\gamma' = 1 - c\ell_{0}^{N_1s}$.

At this stage, we fix the total generation gap $m_1 \approx m_0 + |\log \delta_0| + N_1$.

To establish the bound for $Q$, we decompose the sum over $P \in \mathcal{D}_{ m_1}(Q)$ into those cubes contained in $Q_2$ and those in $Q \setminus Q_2$:
\begin{align*}
    \sum_{P \in \mathcal{D}_{ m_1}(Q)} & \omega(P)^{1/2} \mu(P)^{1/2} \\
    &= \Bigg( \sum_{P \in \mathcal{D}_{ m_1}(Q)\cap \mathcal{D}_{\mu}(Q_2) } + \sum_{P \in \mathcal{D}_{ m_1}(Q) \setminus \mathcal{D}_{\mu}(Q_2)} \Bigg) \omega(P)^{1/2} \mu(P)^{1/2} \\
    &\le \gamma' \omega(Q_2)^{1/2} \mu(Q_2)^{1/2} + \omega(Q \setminus Q_2)^{1/2} \mu(Q \setminus Q_2)^{1/2} \\
    &\le \Big( \omega(Q \setminus Q_2) + \gamma' \omega(Q_2) \Big)^{1/2} \Big( \mu(Q \setminus Q_2) + \gamma' \mu(Q_2) \Big)^{1/2} \\
    &= \Big( \omega(Q) - (1-\gamma')\omega(Q_2) \Big)^{1/2} \Big( \mu(Q) - (1-\gamma')\mu(Q_2) \Big)^{1/2} \\
    &\le \omega(Q)^{1/2} \mu(Q)^{1/2} \Big( 1 - (1-\gamma') \frac{\mu(Q_2)}{\mu(Q)} \Big)^{1/2},
\end{align*}
Here, the second inequality follows from the Cauchy-Schwarz inequality applied to vectors in $\mathbb{R}^2$, and in the final step, we used the trivial bound $\omega(Q) - (1-\gamma')\omega(Q_2) \le \omega(Q)$.

Recall that the Ahlfors regularity implies the  bound $\frac{\mu(Q_2)}{\mu(Q)} \gtrsim \ell_0^{s(m_0+c|\log \delta_0|)}$. Thus, we obtain the desired estimate \eqref{3.17} by setting$$\gamma = \Big( 1 - c'(1-\gamma') \ell_0^{s(m_0+c|\log \delta_0|)} \Big)^{1/2} \in (0,1),$$which concludes the proof.
\end{proof}
With Proposition \ref{prop:stopping1} established, the proof of Theorem \ref{thm:main_planar} is complete. Observe that the decay estimate \eqref{3.17} obtained here is structurally identical to the estimate established in Lemma \ref{lem:decay}, with the generation gap $m_0$ simply replaced by the new parameter $m_1$. Since $m_1$ depends only on the structural constants $\delta_0$ and $C_0$, we can proceed by applying the exact same dimension drop argument used in Section \ref{11.28}. This immediately yields Theorem \ref{thm:main_planar}.

\appendix

\section{Proof of the regularised measure inequality} 
We work with a regularised measure $\nu:=\omega|_{3B_0}*\varphi_t$, where $\varphi_t$ is a standard mollifier at scale $t>0$ and $B_0$ is a fixed ball.

\begin{lemma}\label{lem:beta-inequality}
Fix $x\in B_0$ and assume $2t<r$. Then  
\[
\beta_{2,\nu}^n(x,2r)^2 \gtrsim \beta_{2,\omega}^n(x,r)^2 - \bigl(\tfrac{t}{r}\bigr)^2\Theta^n_\omega(x,3r).
\]
\end{lemma}

\begin{proof}
By definition of $\beta$, it suffices to show that for any $n$-plane $L$ intersecting $B(x,2r)$,  
\[
\Bigl|\int_{B(x,2r)}\operatorname{dist}(y,L)^2\,d\nu(y)-\int_{B(x,r)}\operatorname{dist}(y,L)^2\,d\omega(y)\Bigr|\lesssim t^2\,\omega(B(x,3r)).
\]
Choose a smooth cutoff $\Phi\in C_c^\infty(B(x,2r))$ with  
\[
\Phi\equiv1\ \text{on}\ B(x,r),\qquad |\nabla\Phi|\lesssim 1/r,\qquad |\nabla^2\Phi|\lesssim 1/r^2. \tag{1}
\]
Define $g(y):=\operatorname{dist}(y,L)^2\Phi(y)$. Then  
\[
\Bigl|\int g\,d\nu-\int g\,d\omega\Bigr| = \Bigl|\int g\,d\omega-\int(g*\varphi_t)\,d\omega\Bigr|
\le \int|g-g*\varphi_t|\,d\omega.
\]
For a fixed $z$, write  
\[
(g*\varphi_t)(z)-g(z)=\int[g(z-y)-g(z)]\varphi_t(y)\,dy.
\]  
Using the second‑order Taylor expansion,  
\[
g(z-y)-g(z)=-\nabla g(z)\cdot y+\frac12 y^T\nabla^2g(\xi)y,
\]  
and the evenness of $\varphi_t$ eliminates the linear term. Because $\operatorname{supp}\varphi_t\subset B(0,t)$, we have $|y|\le t$, hence  
\[
|(g*\varphi_t)(z)-g(z)|\lesssim t^2\sup_{\xi\in B(z,t)}\|\nabla^2g(\xi)\|\lesssim t^2,
\]  
where the last estimate uses (1). Since $\operatorname{supp}g\subset B(x,2r)$ and $2t<r$, the difference $g-g*\varphi_t$ is supported in $B(x,2r+t)\subset B(x,3r)$. Integrating the pointwise bound yields the desired estimate.
\end{proof}

\section{Technical estimates for the planar case}

The purpose of this appendix is to establish the technical estimates required for the planar case. The proof follows the same lines as the argument  in \cite[Chapter 7]{T14}. We adapt the original proof, which was established for $1$-Ahlfors regular sets, to the $s$-Ahlfors regular sets. The full details are provided for the reader's convenience.

Assume that $E \subset \mathbb{R}^2$ is an $(s, C_0)$-Ahlfors regular compact set with $s\in(0, 1]$. Recall that for any $Q\in\mathcal{D}_\mu (E),$ we define
 $$\beta_\infty(Q) := \inf_L \sup_{y \in Q} \frac{\operatorname{dist}(y, L)}{\ell(Q)},$$
 \[
 \beta_{1,\mu}^s(Q) := \inf_L  \frac{1}{\ell(Q)^s} \int_{Q}  \frac{\operatorname{dist}(y, L)}{\ell(Q)}  d\mu(y),
 \]
 where in both cases the infimum are taken over all  affine lines $L \subset \mathbb{R}^{2}.$
 
\begin{lemma}\label{16.09} Let $Q,R\in\mathcal{D}_{\mu}(E)$ with $Q\subset R$ and let $L_{R}$ be a line which minimizes $\beta_{1,\mu}^s(R).$
 Then for any $x\in Q,$  the following holds
  \[
  \operatorname{dist}(x,L_R)\lesssim_{C_0}
  \sum_{P\in \mathcal{D}_{\mu}:Q\subset P\subset R}\beta_{1,\mu}^s (P)\ell(P) + \ell(Q).
  \]
\end{lemma}

\begin{proof}Let $n_0\geq 1$ be some fixed positive integer to be chosen below.
    We may assume that  $Q\in \mathcal D_{n_0N_Q}(R)$, where $N_Q$ is some positive integer. 
    Consider the sequence
of dyadic cubes $R = Q_0 \supset Q_1 \supset Q_2 \supset \cdots$ such that $x \in Q_m\in\mathcal D_\mu(E)$ for each $m \ge 1$ and
$\ell(Q_m) = \ell_0^{mn_0}\ell(R)$. Let $\varepsilon_0$ be some (small) constant that will be fixed below
too. Let $N \in [0,N_Q] $ be the least integer such that $\beta_{1,\mu}^s(Q_N) \ge \varepsilon_0$. If $N$ does not
exist because $\beta_{1,\mu}^s(Q_m) < \varepsilon_0$ for $0\le m \le N_Q$, we let $N=N_Q$. Let $a_N$ be any point from $Q_N$ and for $m = N-1, N-2, \dots, 0.$  Let $a_m$ be
the orthogonal projection of $a_{m+1}$ onto a line that minimizes $\beta_{1,\mu}^s(Q_m)$, that we
denote by $L_{Q_m}$. Then we have
\begin{equation}
\operatorname{dist}(a_N, L_R) \le \operatorname{dist}(a_N, L_{Q_N}) +
\sum_{m=0}^{N-1} |a_m - a_{m+1}|. 
\end{equation}
Our next objective consists in showing that
\begin{equation}\label{9.15}
|a_m - a_{m+1}| \lesssim_{C_0}\beta_{1,\mu}^s(Q_m)\,\ell(Q_m) \quad\text{for } m = 0, 1, \dots, N-1. 
\end{equation}
Let us see first that the desired result follows from this estimate. Indeed, from the above inequalities,  we infer that
\[
\operatorname{dist}(a_N, L_R) \lesssim_{C_0} \operatorname{dist}(a_N, L_{Q_N}) +
\sum_{m=0}^{N-1} \beta_{1,\mu}^s(Q_m)\,\ell(Q_m).
\]
If $\beta_{1,\mu}^s(Q_m) < \varepsilon_0$ for $0\le m \le N_Q$ (in this case $N=N_Q$), we let $a_N=x$ in the preceding inequality 
and since $\operatorname{dist}(a_N, L_{Q_N}) \lesssim \ell(Q)$,  Lemma \ref{16.09} follows. If $\beta_{1,\mu}^s(Q_N) \ge \varepsilon_0$, then
\[
|x - a_N| + \operatorname{dist}(a_N, L_{Q_N}) \lesssim \ell(Q_N) \le \varepsilon_0^{-1} \beta_{1,\mu}^s(Q_N)\,\ell(Q_N).
\]
Thus, 
\[
\operatorname{dist}(x, L_R) \le |x - a_N| + \operatorname{dist}(a_N, L_{Q_N}) +
\sum_{m=0}^{N-1} |a_m - a_{m+1}|
\le \varepsilon_0^{-1} \sum_{m=0}^{N} \beta_{1,\mu}^s(Q_m)\,\ell(Q_m).
\]

To prove \eqref{9.15}, we claim that
\begin{claim}\label{23.23}
For $m=0, 1, \dots, N-1,$
\begin{equation}\label{9.25}
    \operatorname{dist}_H\bigl(L_{Q_m} \cap 5B(Q_m),\; L_{Q_{m+1}} \cap 5B(Q_m)\bigr) \lesssim_{C_0} \beta_{1,\mu}^s(Q_m)\,\ell(Q_m).
\end{equation}
\end{claim}

\begin{proof}[Proof of Claim \ref{23.23}]
We define the sets of ``good'' points in $Q_{m+1}$ that are close to the respective best approximating lines:
$$A_1 := \{x \in Q_{m+1} \mid \operatorname{dist}(x, L_{Q_{m+1}}) \le K \beta_{1,\mu}^s (Q_{m+1})\ell(Q_{m+1})\}$$
and
$$A_2 := \{x \in Q_{m+1} \mid \operatorname{dist}(x, L_{Q_{m}}) \le K' \beta_{1,\mu}^s (Q_{m})\ell(Q_{m})\}.$$

By Chebyshev's inequality, we can choose the constants $K$ and $K'$ to be sufficiently large such that
$$\mu(A_1) \ge \frac{3}{4} \mu(Q_{m+1}) \quad \text{and} \quad \mu(A_2) \ge \frac{3}{4} \mu(Q_{m+1}).$$

This yields that their intersection has a large measure:
$$\mu(A_1 \cap A_2) \ge \frac{1}{2} \mu(Q_{m+1}) \gtrsim_{C_0} \ell(Q_{m+1})^s \approx \ell(Q_m)^s.$$

By the Ahlfors-David regularity of $\mu$, the intersection $A_1 \cap A_2$ cannot be concentrated in a ball of arbitrarily small radius. Therefore, we can find two points $x_1, x_2 \in A_1 \cap A_2$ such that $|x_1 - x_2| \approx \ell(Q_m)$.

Since both $x_1$ and $x_2$ are close to $L_{Q_m}$ and $L_{Q_{m+1}}$ within an error of $O(\beta_{1,\mu}^s(Q_m)\ell(Q_m))$, the relative angle and translation between the two lines are bounded by this error. This concludes the proof.
\end{proof}

Then we need to show first that $a_m \in 5B(Q_m)$ for $m = N, N-1, \dots, 1$. We argue
by backward induction. Indeed, for $m = N$, this holds by the definition of $a_N$.
Assume now that $a_{m+1} \in 5B(Q_{m+1})$ and let us see that $a_m \in 5B(Q_{m})$. Remember that
for $m = N-1, N-2, \dots, 1$, we have $\beta_{1,\mu}^s(Q_m) \le \varepsilon_0$. By the Ahlfors regularity of $\mu$,
all points $y \in Q_{m+1} \subset Q_m$ satisfy
\[
\operatorname{dist}(y, L_{Q_m}) \le C(\varepsilon_0)\,\ell(Q_m) \le \ell(Q_{m+1})/2,
\]
assuming that $\varepsilon_0$ has been taken small enough (depending also on the choice of
$n_0$). So we infer that $L_{Q_m} \cap 5B(Q_{m+1}) \neq \varnothing$. Recall that, by the induction hypothesis,
$a_{m+1} \in 5B(Q_{m+1})$. If $n_0$ has been chosen big enough we deduce that
\[
|a_{m+1} - a_m| = \operatorname{dist}(a_{m+1}, L_{Q_m}) \le \operatorname{diam}(5B(Q_{m+1})) \le \ell(Q_m). 
\]
Also, for $n_0$ big enough again, $5B(Q_{m+1}) \subset 2B(Q_{m})$, and thus $a_{m+1} \in 2B(Q_{m})$. This fact
and the above estime imply that $a_m \in 5B(Q_{m})$.

The estimate  follows now easily from  the claim using the fact that $a_m,
a_{m+1} \in 5B(Q_{m})$:
\[
|a_m - a_{m+1}| \le \operatorname{dist}_H\bigl(L_{Q_m} \cap 5B(Q_{m}),\; L_{Q_{m+1}} \cap 5B(Q_{m})\bigr) \lesssim_{C_0} \beta_{1,\mu}^s(Q_m)\,\ell(Q_m).
\]
\end{proof}

\begin{lemma}\label{beta_infty}
Fix a cube $R\in \mathcal{D}_{\mu}(E)$. Then for all integers $N \ge 1$, we have
$$\sum_{Q\in\mathcal{D}_{0,N}(R)} \beta_{\infty}(Q)^2 \ell(Q)^s \lesssim_{C_0} \sum_{Q\in\mathcal{D}_{0,N}(R)} \beta_{1,\mu}^s (Q)^2 \ell(Q)^s +\ell(R)^s.$$
\end{lemma}
\begin{proof}
Fix $Q\in\mathcal{D}_{0,N}(R)$ and let $L_Q$ be a line that minimizes $\beta_{1,\mu}^s(Q)$. By the definition of $\beta_{\infty}(Q)$, we can find 
$x\in Q$ such that
$$\beta_{\infty}(Q)\ell(Q)\le \sup_{y \in Q} \operatorname{dist}(y, L_Q)\le 2\operatorname{dist}(x,L_Q).$$
Let $S_Q\in \mathcal D_N(R)$ with $x\in S_Q\subset Q$, so that by Lemma \ref{16.09}, 
$$\operatorname{dist}(x,L_Q)\lesssim_{C_0} \sum_{P\in \mathcal{D}_{\mu}: S_Q\subset P\subset Q}\beta_{1,\mu}^s (P)\ell(P) + \ell(S_Q).$$
Therefore, for $s\in (0, 1)$, applying the Cauchy-Schwarz inequality yields
$$\begin{aligned}
\beta_{\infty}(Q)^2 \ell(Q)^2 &\lesssim_{C_0} \!\Bigg( \sum_{P\in \mathcal{D}_{\mu}: S_Q\subset P\subset Q} \!\!\beta_{1,\mu}^s(P)^2 \ell(P)^{3s/2} \Bigg) \Bigg( \sum_{P\in \mathcal{D}_{\mu}: S_Q\subset P\subset Q}\!\! \ell(P)^{2-3s/2} \Bigg) +\ell(S_Q)^2 \\
&\lesssim_{C_0} \sum_{P\in \mathcal{D}_{\mu}: S_Q\subset P\subset Q} \beta_{1,\mu}^s(P)^2 \ell(P)^{3s/2} \ell(Q)^{2-3s/2}+\ell_0^{2N}\ell(R)^2.
\end{aligned}$$
Consequently, by interchanging the order of summation, we obtain
\begin{equation}\label{eqfj62}\begin{aligned}
   \sum_{Q\in\mathcal{D}_{0,N}(R)} \beta_{\infty}(Q)^2 \ell(Q)^s &\lesssim_{C_0} \sum_{Q\in\mathcal{D}_{0,N}(R)}
   \sum_{P\in\mathcal{D}_{0,N}(R):P\subset Q}
\beta_{1,\mu}^s(P)^2 \ell(P)^{3s/2} \ell(Q)^{-s/2}\\&\quad
+ \sum_{Q\in\mathcal{D}_{0,N}(R)}\ell(Q)^{s}\,\frac{\ell_0^{2N}\ell(R)^2}{\ell(Q)^2}\\
  & \lesssim_{C_0} \sum_{P\in\mathcal{D}_{0,N}(R)} \beta_{1,\mu}^s(P)^2 \ell(P)^{3s/2}\!\! \sum_{Q:P\subset Q\subset R} \ell(Q)^{-s/2} \\ &\quad + \sum_{Q\in\mathcal{D}_{0,N}(R)}\ell(Q)^{s}\,\frac{\ell_0^{2N}\ell(R)^2}{\ell(Q)^2}.
\end{aligned}
\end{equation}
To deal with the first term on the right hand side, notice that
$$\sum_{Q:P\subset Q\subset R } \ell(Q)^{-s/2} \lesssim \ell(P)^{-s/2}.$$
For the second term, we write
$$\sum_{Q\in\mathcal{D}_{0,N}(R)}\ell(Q)^{s}\,\frac{\ell_0^{2N}\ell(R)^2}{\ell(Q)^2} = \sum_{0\leq k\leq N}
\sum_{Q\in\mathcal{D}_{k}(R)} \ell(Q)^{s}\,\ell_0^{2N-2k}\lesssim_{C_0} \ell(R)^s.$$
Plugging the last estimates into \eqref{eqfj62}, the lemma follows.
\end{proof}




\end{document}